\documentclass[11pt]{amsart}
\usepackage[mathscr]{eucal}
\usepackage{graphicx}
\usepackage{float}
\restylefloat{figure}

\usepackage{fullpage}
\newtheorem{theorem}{Theorem}[section]
\newtheorem{lemma}[theorem]{Lemma}

\newtheorem{corollary}[theorem]{Corollary}
\newtheorem{assumption}[theorem]{Assumption}
\def\NN{\hbox{I\kern-.2em\hbox{N}}}
\def\RR{\hbox{I\kern-.2em\hbox{R}}}

\newtheorem{definition}[theorem]{Definition}
\newtheorem{example}[theorem]{Example}

\newtheorem{remark}[theorem]{Remark}

\date{Last update \today}
\begin{document}

\title[Predictive control with noise]{Stabilisation of difference 
equations with noisy prediction-based control}
\author[E. Braverman, C. Kelly and A. Rodkina]
{E. Braverman, C. Kelly and A. Rodkina}

\address{Department of Mathematics, University of Calgary, Calgary, Alberta T2N1N4, Canada}
\email{maelena@math.ucalgary.ca}
\address{Department of Mathematics \\
The University of the West Indies, Mona Campus, Kingston, Jamaica} 
\email{conall.kelly@uwimona.edu.jm}
\address{Department of Mathematics  \\
The University of the West Indies, Mona Campus, Kingston, Jamaica}
\email{alexandra.rodkina@uwimona.edu.jm}

\thanks{The first author was supported by the NSERC grant RGPIN-2015-05976, 
all the authors were supported by American Institute of Mathematics SQuaRE program}
\thanks{Corresponding author email: {\tt conall.kelly@uwimona.edu.jm}}

\begin{abstract}
We consider the influence of stochastic perturbations on stability of a unique positive equilibrium 
of a difference equation subject to prediction-based control. 
These perturbations may be multiplicative 
$$x_{n+1}=f(x_n)-\left( \alpha + l\xi_{n+1} \right) (f(x_n)-x_n), \quad n=0, 1, \dots$$
if they arise from stochastic variation of the control parameter, or additive 
$$x_{n+1}=f(x_n)-\alpha(f(x_n)-x_n) +l\xi_{n+1}, \quad n=0, 1, \dots $$
if they reflect the presence of systemic noise.

We begin by relaxing the control parameter in the deterministic equation, and deriving a range of values for the parameter over which all solutions eventually enter an invariant interval. 
Then, by allowing the variation to be stochastic, we derive sufficient conditions (less restrictive than known ones for the unperturbed equation) 
under which the positive equilibrium will be globally a.s. asymptotically stable: 
i.e. the presence of noise improves the known effectiveness of prediction-based control.  
Finally, we show that systemic noise has a ``blurring'' effect on the positive equilibrium, which can be made arbitrarily small by 
controlling the noise intensity.  
Numerical examples illustrate our results.




\end{abstract}

\maketitle

\section{Introduction}
\label{sec:intr}

The dynamics of discrete maps can be complicated, and various methods may be introduced 
to control their asymptotic behaviour. In addition, both the intrinsic dynamics and the control may involve stochasticity.

We may ask the following of stochastically perturbed difference equations:
\begin{enumerate}
\item
If the original (non-stochastic) map has chaotic or unknown dynamics, can we stabilise the equation
by introducing a control with a stochastic component?
\item
If the non-stochastic equation is either stable or has known dynamics (for example, a stable two-cycle \cite{BR_2012}), do those dynamics persist when a stochastic perturbation is introduced?
\end{enumerate}
In this article, we consider both these questions in the context of prediction-based control (PBC, or predictive control).
Ushio and Yamamoto \cite{uy99} introduced PBC as a method of stabilising unstable periodic orbits of 
\begin{equation}
\label{eq:intr1}
x_{n+1}=f(x_n), \quad x_0>0, \quad n\in {\mathbb N}_0,
\end{equation}
where ${\mathbb N}_0=\{0,1,2, \dots, \}$.
The method overcomes some of the limitations of delayed feedback control (introduced by Pyragas~\cite{pyr}), 
and does not require the a priori approximation of periodic orbits, as does the OGY method developed by Ott 
et al~\cite{OGY}. 

The general form of PBC is 
$$x_{n+1}=f(x_n)-\alpha(f^k(x_n)-x_n), \quad x_0>0, \quad n\in {\mathbb N}_0,
$$
where $\alpha \in (0,1)$ and $f^k$ is the $k$th iteration of $f$. If $k=1$, PBC becomes
\begin{equation}
\label{eq:intr2}
x_{n+1}=f(x_n)- \alpha (f(x_n)-x_n)= (1-\alpha)f(x_n)+ \alpha x_n, \quad x_0>0, \quad n\in {\mathbb N}_0.
\end{equation}

Recently, it has been shown how PBC can be used to manage population size via population reduction by ensuring that the positive equilibrium of a class of one-dimensional maps commonly used to model population dynamics is globally asymptotically stable after the application of the control~\cite{FL2010}. Similar effects are also possible if it is not feasible to apply the control at every timestep. This variation on the technique is referred to as PBC-based pulse stabilisation~\cite{BerLizCAMWA,LizPotsche14}.

Here, we investigate the influence of stochastic perturbations on the ability of PBC to induce global asymptotic stability of a positive point equilibrium of a class of equations of the form \eqref{eq:intr1}. It is reasonable to introduce noise in one of two ways. First, the implementation of PBC relies upon a controlling agent to change the state of the system in a way characterised by the value of the control parameter $\alpha$. In reality we expect that such precise control is impossible, and the actual change will be characterised by a control sequence $\{\alpha_n\}_{n\in\mathbb{N}_0}$ with terms that vary randomly around $\alpha$ with some distribution. 
This will lead to a state-dependent, or multiplicative, stochastic perturbation. 
Second, the system itself may be subject to extrinsic noise, which may be modelled by a state-independent, or additive, perturbation.

The fact that stochastic perturbation can stabilise an unstable equilibrium has been understood since the 1950s: consider the well-known example of the pendulum of Kapica \cite{Kapica}. More recently, a general theory of stochastic stabilisation and destabilisation of ordinary differential equations has developed from \cite{Mao}: a comprehensive review of the literature is presented in \cite{AMR1}. This theory extends to functional differential equations: for example \cite{Appleby,AKMR2} and references therein. 

Stochastic stabilisation and destabilisation is also possible for difference equations; see for example \cite{ABR,AMR06}. However, the qualitative behaviour of stochastic difference equations may be dramatically different from that seen in the continuous-time case, and must be investigated separately. For example, in \cite{KR09}, solutions of a nonlinear stochastic difference equation with multiplicative noise arising from an Euler discretisation of an It\^o-type SDE are shown to demonstrate monotonic convergence to a point equilibrium with high probability. This behaviour is not possible in the continuous-time limit. 

Now, consider the structure of the map $f$. We impose the Lipschitz-type assumption on the function $f$ around 
the unique positive equilibrium $K$.
\begin{assumption}
\label{as:slope}
$f:[0,\infty) \to [0,\infty)$ is a continuous function, $f(x)>0$ for $x>0$,
$f(x)>x$ for $x \in (0,K)$, $f(x)<x$ for $x >K$, and there exists $M \geq 1$ such that 
\begin{equation}
\label{eq:Mcond}
|f(x)-K| \leq  M |x-K|.
\end{equation}
\end{assumption}
Note that under Assumption \ref{as:slope} function $f$ has only a single positive point equilibrium $K$. We will also suppose that $f$ is decreasing on an interval that includes $K$:
\begin{assumption}
\label{as:3}
There is a point $c<K$ such that $f(x)$ is monotone decreasing on $[c,\infty)$.
\end{assumption}

It is quite common for Assumptions~\ref{as:slope} and \ref{as:3} to hold for models of population dynamics, 
and in particular for models characterised by a unimodal map: we illustrate this with Examples 
\ref{ex:Ric}-\ref{ex:BH}. It follows from Singer~\cite{Singer} that, when additionally $f$ has a negative Schwarzian derivative $(Sf)(x)=f'''(x)/f'(x) - \frac{3}{2}(f''(x)/f'(x))^2<0$, the equilibrium $K$ is globally asymptotically stable if and only if it is locally asymptotically stable. In each case, as the system parameter grows, a stable cycle replaces a stable equilibrium which loses its stability, there are period-doubling bifurcations and eventually chaotic behaviour.

\begin{example}
\label{ex:Ric}
For the Ricker model
\begin{equation}
\label{eq:ricker}
x_{n+1} = x_n e^{r(1-x_n)}, \quad x_0>0, \quad n\in {\mathbb N}_0,
\end{equation}
Assumptions~\ref{as:slope} and \ref{as:3} both hold with $K=1$, and the global maximum is attained at $c=1/r<K=1$ for $r>1$. 
Let us note that for $r \leq 1$ the positive equilibrium is globally asymptotically stable and the convergence of solutions to $K$ is monotone.
However, for $r>2$ the equilibrium becomes unstable. 
\end{example}

\begin{example}
\label{ex:log}
The truncated logistic model 
\begin{equation}
\label{eq:logistic}
x_{n+1} = \max\left\{ r x_n (1-x_n), 0 \right\}, \quad x_0>0, \quad n\in {\mathbb N}_0,
\end{equation}
with $r>1$ and $c=\frac{1}{2} <K =1-1/r$, also
satisfies Assumptions~\ref{as:slope} and \ref{as:3}.
Again, for $r \leq 2$, the equilibrium $K$ is globally
asymptotically stable, with monotone convergence to $K$, while for $r>3$ the equilibrium $K$ is unstable.
\end{example}

\begin{example}
\label{ex:BH}
For the modifications of the Beverton-Holt equation
\begin{equation}
\label{eq:BHm1}
x_{n+1} = \frac{Ax_n}{1+Bx_n^{\gamma}},\quad A>1,~B>0,~\gamma >1, , \quad x_0>0, \quad n\in {\mathbb N}_0, 
\end{equation}
and
\begin{equation}
\label{eq:BHm2}
x_{n+1} = \frac{Ax_n}{(1+Bx_n)^{\gamma}},\quad A>1,~B>0,~\gamma>1, \quad x_0>0, \quad n\in {\mathbb N}_0
\end{equation}
Assumption~\ref{as:slope} holds. Also, \eqref{eq:BHm1} and \eqref{eq:BHm2} satisfy Assumption~\ref{as:3} as 
long as the 
point at which the map on the right-hand side takes its maximum value is less than that of the point equilibrium. 
If Assumption~\ref{as:3} is not satisfied, the function is monotone increasing up to the unique positive 
point equilibrium, and thus all solutions converge to the positive equilibrium,
and the convergence is monotone.
If all $x_n>K$, we have a monotonically decreasing sequence. 
If we fix $B$ in \eqref{eq:BHm1} and \eqref{eq:BHm2} and consider the growing $A$, the equation loses stability and experiences transition to chaos through a series of period-doubling bifurcations.
\end{example}
The article has the following structure. In Section 2 we relax the control parameter $\alpha$, replacing it with the variable control sequence $\{\alpha_n\}_{n\in\mathbb{N}_0}$, and yielding the equation
\begin{equation}
\label{eq:intr3}
x_{n+1}=f(x_n)- \alpha_n (f(x_n)-x_n)= (1-\alpha_n)f(x_n)+ \alpha_n x_n, \quad x_0>0, \quad n\in {\mathbb N}_0.
\end{equation}
We identify a range over which $\{\alpha_n\}_{n\in\mathbb{N}_0}$ may vary deterministically while still ensuring the global asymptotic stability of the positive equilibrium $K$. We confirm that, without imposing any constraints on the range of values over which the control sequence $\{\alpha_n\}_{n\in\mathbb{N}_0}$ may vary, there exists an invariant interval, containing $K$, under the controlled map. We then introduce constraints on terms of the sequence $\{\alpha_n\}_{n\in\mathbb{N}_0}$ which ensure that all solutions will eventually enter this invariant interval. 

In Section 3, we assume that the variation of $\alpha_n$ around $\alpha$ is bounded and stochastic, which 
results in a PBC equation with multiplicative noise of intensity $l$. After identifying constraints on 
$\alpha$ and $l$ under which a domain of local stability for $K$ exists for all trajectories, we 
demonstrate that the presence of an appropriate noise perturbation in fact ensures that almost all trajectories will eventually enter this domain of local stability, hence providing global a.s. asymptotic stabilisation of $K$. The known range of values of $\alpha$ under which this stabilisation occurs is larger than for the deterministic PBC equation, and in this sense the stochastic perturbation improves the stabilising properties of PBC.

In Section 4, we suppose that the noise is acting systemically rather than through the control parameter, 
which results in a PBC equation with an additive noise. In this setting it is possible to show that, under 
certain conditions on the noise intensity $l$, the noise causes a ``blurring'' of the positive equilibrium $K$ in the sense that the controlled solutions will enter and remain within a neighbourhood of $K$, and the size of that neighbourhood can be made arbitrarily small by an appropriate choice of $l$.

Finally, Section 5 contains some simulations that illustrate the results of the article, and a brief summary.

\section{Deterministic Equations with Variable PBC}
\label{sec:det}
We begin by relaxing the control variable in the deterministic PBC equation \eqref{eq:intr2}, both as a generalisation to equations of form \eqref{eq:intr3} and to support our analysis of the system with stochastically varying control in Section \ref{sec:3}. Deterministic PBC equation \eqref{eq:intr3} with variable control parameter may be written in the form
\begin{equation}
\label{eq:Fa}
x_{n+1}=F_{\alpha_n}(x_n),\quad x_0>0, \quad n\in {\mathbb N}_0,
\end{equation}
where
\begin{equation}
\label{eq:F}
F_{\alpha}(x):= \alpha x+ (1-\alpha)f(x).
\end{equation}

The following result extends \cite[Theorem 2.2]{BerLizCAMWA} to develop conditions on the magnitude of variation of $\alpha_n$ 
for solutions of \eqref{eq:Fa} to approach the positive equilbrium $K$ at some minimum rate.
\begin{lemma} 
\label{lem:PBC}
Let Assumption~\ref{as:slope} hold and each $\alpha_n$ satisfy $\alpha_n \in [a,1)$,
where  $$a \in \left( 1-\frac{1}{M},1 \right).$$ 
Let $\{x_n\}_{n\in\mathbb{N}_0}$ be any solution of \eqref{eq:Fa} with $x_0>0$. Then
\begin{enumerate}
\item[(i)]
the sequence $\{|x_n-K|\}_{n \in {\mathbb N}_0}$ is non-increasing;
\item[(ii)]
If there is $b \in (0,1)$ for which $\alpha_n \leq b <1$,
for any $x_0>0$
\[
\lim_{n\to\infty}x_n=K;
\]
\item[(iii)]
If in addition Assumption~\ref{as:3} holds, there exists $n_0 \in {\mathbb N}_0$ such that  $x_n \geq c$ for $n\geq n_0$ 
and
\begin{equation}
\label{decay}
\left| x_{n+1}-K \right| = \left| F_{\alpha_n}(x_{n})-K \right| \leq \gamma \left| x_{n}-K \right|,
\end{equation}
where
\begin{equation}
\label{gamma}
\gamma = \max \{ b, 1-a \}.
\end{equation}
\end{enumerate}
\end{lemma}
\begin{proof}
We address each part in turn.
\begin{itemize}
\item[Part (i):]
First, we prove convergence in the case where the signs of $\{x_n-K\}_{n\in\mathbb{N}_0}$ are eventually constant: solutions eventually remain either above or below the positive equilibrium $K$. Suppose that there exists $n_1 \in {\mathbb N}_0$ such that  $x_j<K$ for $j \geq n_1$. Then the subsequence $\{x_n\}_{j\geq n_1}$ is monotone increasing, since by Assumption \ref{as:slope} and \eqref{eq:F},
\[
x_j<K\,\,\Rightarrow\,\,f(x_j)>x_j\,\,\Rightarrow\,\, F_{\alpha_j}(x_j)>x_j\,\,\Rightarrow\,\,
x_{j+1}>x_j.
\]

Next, we consider the case when the terms of $\{x_n-K\}_{n\in\mathbb{N}_0}$ change signs infinitely often. 
Note that we need to take into consideration only the indices $i$ where $x_i<K$ and $F_{\alpha_i}(x_i)>K$ 
or $x_{i}>K$ and $F_{\alpha_{i}}(x_{i}) < K$. At any $n$ where $(x_n-K)(x_{n+1}-K)>0$,
we have $|x_{n+1}-K|<|x_n-K |$. 
Subsequences of $\{x_n\}_{n\in\mathbb{N}_0}$ that do not switch in this way will approach $K$ monotonically, 
as proven above. We must prove that $|x_{i+1}-K|<|x_i-K|$ at these switches as well.   

Suppose first that $x_i \in (0, K)$ and $F_{\alpha_i}(x_i) > K$. Then $f(x_i)>K$ necessarily, since otherwise
$$
F_{\alpha_i}(x_i)=\alpha_i x_i+(1-\alpha_i)f(x_i)
\le \alpha_i K+(1-\alpha_i)K=K.
$$
It is also the case that $x_{i+1} \in (K,f(x_i))$, since
$$
x_{i+1}=F_{\alpha_i}(x_i)
\le \alpha_i f(x_i)+(1-\alpha_i)f(x_i)
= f(x_i).
$$

Note that $a \in (1-1/M,1)$ implies $1-a < \frac{1}{M}$.  Since $\alpha_i>a$, it follows that $(1-\alpha_i)M<1$, and we have from $F_{\alpha_i}(x_i) > K$ and \eqref{eq:F} that
\begin{eqnarray*} 
\left| x_{i+1}-K \right| & = &   F_{\alpha_i}(x_{i})-K   \\
& = &   (1-\alpha_i) (f(x_i)- K) - \alpha_i (K-x_i)  \\
& \leq &  (1-\alpha_i) M(K-x_i) - \alpha_i (K-x_i) \\
& \leq &  |x_i-K| - \alpha_i |x_i-K| \\ & \leq &  (1-a) \left| x_{i}-K \right|.
\end{eqnarray*}
By similar reasoning, if $x_i > K$ and $F_{\alpha_i}(x_i) < K$ we have $f(x_i) < K$ and $x_{i+1} \in (f(x_i),K)$. Therefore
\begin{eqnarray*} 
\left| x_{i+1}-K \right| & = & K -  F_{\alpha_n}(x_{i}) \\ & =  &  
 (1-\alpha_i) (K-f(x_i)) - \alpha_i (x_i-K)  \\
& \leq &  (1-\alpha_i) M(x_i-K) + \alpha_i (x_i-K) \\
& \leq &  (1-\alpha_i) \left| x_{i}-K \right| \\ & \leq &  (1-a) \left| x_{i}-K \right|.
\end{eqnarray*}
Thus, $\{|x_n-K|\}_{n \in {\mathbb N}_0}$ is a non-increasing sequence, and Part (i) of the statement of the lemma is verified.

\item[Part (ii):]
$\{|x_n-K|\}_{n \in {\mathbb N_0}}$ is a decreasing positive sequence if no terms of the sequence $\{x_n\}_{n\in\mathbb{N}_0}$ coincide with $K$. 
If $x_n < K$ for all $n\geq j$ then $\lim\limits_{n\to\infty}x_n=L>0$. This implies in turn that the left-hand side of
$$
x_{n+1}-x_n=(1-\alpha_n)(f(x_n)-x_n)
$$
tends to zero, and so the right-hand side also tends to zero.
From $1-\alpha_n \geq 1-b>0$ and continuity of $f$, we have $f(L)=L$, so the limit $L$
can only be $K$. The case where $x_j>K$, for all $j\geq n_1$ is treated similarly.
If $(x_{n+1}-K)(x_n-K)<0$ then $|x_{n+1}-K|\leq (1-a) |x_n-K|$, which implies 
\begin{equation}
\label{add_star}
\lim_{n \to \infty} \left| x_n-K \right| =0. 
\end{equation}
Therefore $\lim\limits_{n \to \infty}  x_n = K$, and Part (ii) of the statement of the lemma is confirmed. 
\item[Part (iii):]

Let Assumption~\ref{as:3} hold. By \eqref{add_star}, for any $c>0$ there exists some $n_0 \in {\mathbb N_0}$ such that, for $n\geq n_0$, $|x_{n}-K | \leq  K-c$, and  thus $x_{n} \in [c,\infty)$. Further we consider only $n \geq n_0$.
Also, it has been established above that under the common conditions holding for 
Parts (i)-(iii) in the statement of the lemma, $|x_n-K|$ is decreasing. Let $i$ be an index where a switch 
across the equilibirum $K$ occurs, i.e. $(x_i-K)(x_{i+1}-K)<0$. Then, from the analysis above,
\[
|x_{n+1}-K| \leq (1-a) |x_n-K| \leq \gamma |x_n-K|.
\]  
If $(x_n-K)(x_{n+1}-K)=0$, then $x_j=K$ for all $j>n$, so \eqref{decay} is satisfied in this situation. 
It remains to consider the case where $(x_n-K)(x_{n+1}-K)>0$, $n \geq n_0$.
Suppose first that $x_n<K$, $x_{n+1}<K$ for some $n\geq n_0$. Then $x_n \in [c,K]$, $f(x_n)>K$ and, as $\alpha_n \leq b \leq \gamma$,
\begin{eqnarray*} 
\left| x_{n+1}-K \right| & = & K -  F_{\alpha_n}(x_{n}) =   
(1-\alpha_n) (K-f(x_n)) + \alpha_n (K-x_n)  \\
& < &   \alpha_n (K-x_n) \leq b  (K-x_n) \leq \gamma |K-x_n|.
\end{eqnarray*}
When $x_n > K$, $x_{n+1}>K$ for some $n\geq n_0$, we have $f(x_n)<K$ and
\begin{eqnarray*} 
\left| x_{n+1}-K \right| & = & F_{\alpha_n}(x_{n}) -K 
=  (1-\alpha_n) (f(x_n)-K) + \alpha_n (x_n-K)  \\
& < &   \alpha_n (x_n-K) \leq b  (x_n-K) \leq \gamma |x_n-K|.
\end{eqnarray*}
Taking these results in aggregate we can conclude that $|x_{n+1}-K| \leq \gamma |x_n-K|$ for any $n\geq n_0$,
where $\gamma$ is defined in \eqref{gamma}.
\end{itemize}
\end{proof}

In the case of $\alpha_n \equiv \alpha$, we obtain the following corollary, highlighting the existence of an invariant interval under the map $F_\alpha$ when the control parameter $\alpha$ is constant.

\begin{corollary}
\label{new_add1}
If Assumptions~\ref{as:slope} and \ref{as:3} hold and $\displaystyle \alpha \in \left( 1 - \frac{1}{M}, 1 \right)$ 
then $F_{\alpha}$ defined by \eqref{eq:F} maps the interval $[c,2K - c]$ into $(c,2K-c)$, and  satisfies 
\begin{equation}
\label{contraction}
|F_{\alpha}(x)-K| \leq \gamma |x-K|, \quad x \geq c, 
\end{equation}
with 
\begin{equation}
\label{def_gamma}
\gamma=\max\{ \alpha, 1-\alpha \}.
\end{equation}
\end{corollary}
\begin{proof}
By Lemma~\ref{lem:PBC}, Part (i), 
$$|F_{\alpha}(x)-K| \leq |x-K|, \quad x>0.$$
Thus,  $F_{\alpha}$ maps the interval $[d,2K - d]$ into $(d,2K-d)$
for any $d<K$, in particular, $F_{\alpha}: [c,2K - c]\to (c,2K-c)$.
Next, let $x\geq c$. By Lemma~\ref{lem:PBC},
Part (iii), for $x \in [c,2K - c]$,
$$\left|F_{\alpha} (x) - K\right| \leq \gamma |x-K|,$$
where $\gamma$ defined in \eqref{gamma} takes the form of $\gamma=\max\{ \alpha, 1-\alpha \}$.
\end{proof}

Our next step is prove the existence of an invariant interval under the PBC map when the control parameter $\alpha_n$ is allowed to vary on the interval $[0,1]$. Note first that if Assumptions~\ref{as:slope} and \ref{as:3} hold, the maximum of $f(x)$ on $[0,\infty)$ is attained on $[0,c]$. We construct the endpoints of the invariant interval as follows.
\begin{definition}\label{def:mus}
Under Assumptions~\ref{as:slope} and \ref{as:3}, define
\begin{enumerate}
\item $\mu_0 \in [0,c]$ to be the smallest point where the maximum of $f$ is attained:
\begin{equation}
\label{eq:mu0}
\mu_0 := \inf \left\{ x \in [0,c] \left| f(x)= \max_{s \in [0,\infty)} f(s) \right\} \right. ;
\end{equation}
\item $\mu_2$ to be the value of this maximum:
\begin{equation}
\label{eq:mu2}
\mu_2 := f(\mu_0) \geq f(c)>f(K)=K>c, \quad f(x) \leq \mu_2, \quad x \in [0,\infty);
\end{equation}
\item $\mu_1$ to be the image of $\mu_2$ under $f$: 
\begin{equation}\label{eq:mu1def}
\mu_1:= f(\mu_2).
\end{equation}
\end{enumerate}
\end{definition}
\begin{remark}
By Assumption~\ref{as:3}, $f$ decreases on $[c,\infty)$ and $\mu_2>K>c$, thus   
\begin{equation}
\label{eq:mu1}
\quad \mu_1 =  f(\mu_2) < f(K)=K.
\end{equation}
\end{remark}

\begin{lemma}
\label{lemma_mu}
Suppose that Assumptions~\ref{as:slope} and \ref{as:3} hold, and let $F_{\alpha_n}$ be the PBC map defined in \eqref{eq:F}. For any $\alpha_n \in [0,1]$,
$$F_{\alpha_n}\left([\mu_1,\mu_2]\right)\subseteq [\mu_1,\mu_2].$$
\end{lemma}
\begin{proof}
First, we prove that $f\left([\mu_1,\mu_2]\right)\subseteq [\mu_1,\mu_2]$.
By Parts (1) and (2) of Definition \ref{def:mus}, we have $f(x) \leq \mu_2$ for any $x \in [\mu_1,\mu_2]$. Since $K\in[\mu_1,\mu_2]$, we consider the subintervals $[\mu_1,K]$ and $[K,\mu_2]$ in turn. 
If $x \in [\mu_1,K]$, then $f(x)>x \geq \mu_1$. If $x \in [K,\mu_2]$, due to the fact that $f(x)$ is decreasing on this interval,
$f(x) \geq f(\mu_2)=\mu_1$. Thus, $f\left([\mu_1,\mu_2]\right)\subseteq [\mu_1,\mu_2]$. Furthermore, for any $x \in [\mu_1,\mu_2]$,
$$
F_{\alpha_n}(x) = \alpha_n x + (1-\alpha_n)f(x)
\geq \alpha_n \mu_1+(1-\alpha_n) \mu_1 = \mu_1,
$$
and
$$
F_{\alpha_n}(x) \leq \alpha_n \mu_2 +(1-\alpha_n) \mu_2
= \mu_2.
$$
We conclude that $F_{\alpha_n}\left([\mu_1,\mu_2]\right)\subseteq [\mu_1,\mu_2]$, as required.
\end{proof}

The final result in this section shows that terms of the sequence $\{\alpha_n\}_{n\in\mathbb{N}_0}$ may be constrained in such a way that solutions of the PBC equation \eqref{eq:Fa} eventually enter, and therefore remain, within the interval $[\mu_1,\mu_2]$. We will use this approach to obtain global stochastic stability conditions later in the article.

\begin{lemma}
\label{lemma_add2}
Suppose that Assumptions~\ref{as:slope} and \ref{as:3} hold and there exists $\delta \in \left(0, \frac{1}{2}\right)$ such that 
\begin{equation}
\label{alphadelta}
\alpha_n \in (\delta,1-\delta)
\end{equation}
for every  $n \in {\mathbb N_0}$. Then, for each $x_0>0$, there exists $n_0 \in {\mathbb N_0}$
such that the solution $x_n$ of \eqref{eq:Fa} satisfies $x_n \in [\mu_1,\mu_2]$ for $n \geq n_0$, where $\mu_1$ and $\mu_2$ are defined in \eqref{eq:mu1def} and \eqref{eq:mu2}, respectively.
\end{lemma}
\begin{proof}
We will show that for $x_0 \leq \mu_1$ (and similarly for $x_0 \geq \mu_2$), there exists
$n_0 \in {\mathbb N_0}$ such that  $x_{n_0} \in [\mu_1,\mu_2]$. It will then follow from Lemma \ref{lemma_mu} that $x_{n_0+j}\in [\mu_1,\mu_2]$ for any $j \in {\mathbb N_0}$.

Proceed by contradiction. Suppose first that $x_0 \leq \mu_1$. Then $f(x_0)>x_0$, and
$$
x_1=\alpha_0 x_0 + (1-\alpha_0)f(x_0) > x_0.
$$
Repeating this step confirms that $\{x_n\}_{n\in\mathbb{N}_0}$ is an increasing sequence as long as $x_n<K$. Thus, either there exists $n_0 \in {\mathbb N_0}$ such that  $x_{n_0} > \mu_1$
or $x_n$ is a bounded increasing sequence which has a limit $d \leq \mu_1<K$. In the latter case, obviously $f(d)> d$.
Denote
$$
\Delta := \frac{\delta}{2}(f(d)-d),
$$
where $\delta$ is as defined in \eqref{alphadelta}.
Here $\Delta>0$, so by continuity of $f$ and convergence of $\{x_n\}_{n\in\mathbb{N}_0}$ to $d$, there is $n_1 \in {\mathbb N_0}$ such that for any $n \geq n_1$, 
$$
f(x_n)-x_n > \frac{\Delta}{\delta}.
$$
Then, for any $n \geq n_1$,
\begin{eqnarray*}
x_{n+1}&=& \alpha_n x_n + (1-\alpha_n)f(x_n)=x_n+(1-\alpha_n)(f(x_n)-x_n)\\
&>& x_n+ \frac{1-\alpha_n}{\delta} \Delta > x_n+\Delta.
\end{eqnarray*}
So $x_{n_1+j} \geq \mu_1$ for any $\displaystyle j \geq n_1+\mu_1/\Delta$, which contradicts our assumption that all $x_n \leq \mu_1$. Moreover, by Definition \ref{def:mus}, $f(x) \leq \mu_2$, so that $x_n \leq \mu_2$ for all $n\in\mathbb{N}_0$. We conclude that there exists $n_0 \in {\mathbb N_0}$ such that  $x_{n_0} \in [\mu_1,\mu_2]$.
 
Suppose next that $x_0>\mu_2>K$. If there is an $n_2 \in {\mathbb N_0}$ such that $x_{n_2} < \mu_2$, then either we revert to the previous case or $x_{n_2} \in [\mu_1,\mu_2]$. Otherwise,
$\{x_n\}_{n\in\mathbb{N}_0}$ is a decreasing sequence with a limit $d \geq \mu_2>K$, where $f(d)<d$; moreover, $f(d)<K$, as $f$ is decreasing on $[c,\infty)$, $c<K$.

As before, by the continuity of $f$ and convergence of $\{x_n\}_{n\in\mathbb{N}_0}$ to $d$, for some $\Delta>0$ there exists $n_2 \in {\mathbb N_0}$ such that 
$$x_n-f(x_n)> \frac{\Delta}{\delta}, \quad n \geq n_2.
$$
Then, for any $n \geq n_2$,
$$
x_{n+1}= \alpha_n x_n + (1-\alpha_n)f(x_n)=x_n-(1-\alpha_n)(x_n-f(x_n)) < x_n - \frac{1-\alpha_n}{\delta} \Delta < x_n-\Delta.
$$ 
So there exists $j\in\mathbb{N}_0$ such that $x_j \leq \mu_2$. We conclude that, for all solutions $\{x_n\}_{n\in\mathbb{N}_0}$ of \eqref{eq:Fa}, there exists $n_0 \in {\mathbb N_0}$ such that $x_n \in [\mu_1,\mu_2]$ for $n \geq n_0$.
\end{proof}


\section{Multiplicative Noise}
\label{sec:3}

Let $(\Omega, {\mathcal{F}}, \{\mathcal{F}_n\}_{n \in \mathbb{N}_0}, {\mathbb{P}})$ be a complete, filtered probability space, and let $\{\xi_n\}_{n\in\mathbb{N}_0}$ be a sequence of independent and identically distributed random variables  with common density function $\phi_n$. The filtration $\{\mathcal{F}_n\}_{n \in\mathbb{N}_0}$ is naturally generated by this sequence: $\mathcal{F}_{n} = \sigma \{\xi_{i} : 1\leq i\leq n\}$, for $n\in\mathbb{N}_0$. Among all sequences $\{x_n\}_{n \in N}$ of random variables we consider those for which $x_n$ is $\mathcal{F}_n$-measurable for all $n \in \mathbb{N}_0$. We use the standard abbreviation ``a.s.'' for the wordings ``almost sure'' or ``almost surely'' with respect to ${\mathbb{P}}$. 

In this section, we allow the control parameter $\alpha$ to vary stochastically, by setting $\alpha_n=\alpha+l\xi_{n+1}$ for each $n\in\mathbb{N}_0$, where $l$ controls the intensity of the perturbation and the sequence $\{\xi_n\}_{n\in\mathbb{N}_0}$ additionally satisfies the following assumption.
\begin{assumption}
\label{as:chi1}
Let  $\{\xi_n\}_{n\in\mathbb{N}_0}$    be a sequence of independent and  identically  distributed  
continuous random variables with common density function $\phi$ supported on the interval $[-1,\nu]$, for some $\nu\geq 1$.
\end{assumption}
\begin{remark}
The support of each $\xi_n$ is asymmetric if $\nu>1$, allowing for perturbations where the potential magnitude in the positive direction is larger than that in the negative direction. Note that the possibility that $\mathbb{E}\xi_n=0$ is not ruled out in that case.
\end{remark}
This leads to the following PBC equation with stochastic control
\begin{equation}
\label{eq:multi}
x_{n+1}= \max \left\{ f(x_n)-\left(\alpha + l\xi_{n+1} \right) (f(x_n)-x_n), 0 \right\}, \quad x_0>0, \quad n\in {\mathbb N}_0.
\end{equation}
The right-hand side is truncated to ensure that physically unrealistic negative population sizes cannot occur.




Our first result in this section applies when the perturbation support is symmetric, and provides a bound on the stochastic intensity $l$ that will ensure the convergence of all solution trajectories to the positive equilibrium $K$. A minimum asymptotic convergence rate is also determined.

\begin{theorem}
\label{th_multi}
Let $\{x_n\}_{n\in\mathbb{N}_0}$ be any solution of equation \eqref{eq:multi} with $x_0>0$. Suppose that Assumptions~\ref{as:slope} and \ref{as:chi1} hold, the latter with $\nu=1$.
If 
\[
\alpha \in \left(1-\frac{1}{M},1\right)
\] 
and 
\begin{equation}
\label{eq:noise}
l < \min\left\{ \alpha -\left(1-\frac{1}{M}\right), \,1-\alpha \right\},
\end{equation}
then for all $\omega\in\Omega$
\begin{equation}\label{eq:xconv}
\lim_{n\to\infty}x_n(\omega)=K,\quad x_0>0.
\end{equation}

If in addition Assumption~\ref{as:3} holds, then there exists a finite random number $n_0(\omega)$ such that for $n\geq n_0(\omega)$, $\{x_n(\omega)\}_{n\in\mathbb{N}_0}$ 
satisfies \eqref{decay} with 
\begin{equation}
\label{eq:beta}
\gamma =  \max\left\{ 1- \alpha + l, \, \alpha+l \right\}
\end{equation}
for all $\omega\in\Omega$.
\end{theorem}
\begin{proof}
If \eqref{eq:noise} and Assumptions ~\ref{as:slope},  \ref{as:chi1} ($\nu=1$) hold,  
then for any $\omega \in \Omega$, 
$$\alpha+l \xi_n(\omega) \geq \alpha-l > 1 - \frac{1}{M} \quad \mbox{ and } \quad \alpha+l \xi_n(\omega) \leq \alpha+l <1 .$$
Lemma~\ref{lem:PBC} implies that \eqref{eq:xconv} holds for all $\omega\in\Omega$.

If in addition  Assumption~\ref{as:3} holds, then for all $\omega\in\Omega$, by Lemma~\ref{lem:PBC}, 
there exists $n_0(\omega)\in \mathbb R$ such that, for $n\ge n(\omega)$, we have $x_n(\omega) \geq c$ and 
$$|x_{n+1}(\omega)-K| \leq \gamma |x_n(\omega)-K|,$$
where 
$$\gamma = \max\left\{ 1- (\alpha-l), \alpha+l \right\} =
\max\left\{  1- \alpha + l,\alpha+l \right\},$$
which concludes the proof.
\end{proof}

The Lipschitz-type condition with global constant $M$ given by \eqref{eq:Mcond} in Assumption \ref{as:slope} implies a local Lipschitz-type condition: for any $\varepsilon\in(0,K)$, there exists $M_\varepsilon\leq M$ such that
\begin{equation}
\label{eq:M1cond}
|f(x)-K| \leq  M_\varepsilon |x-K|, \quad x \in (K-\varepsilon,K+\varepsilon).
\end{equation}
In practice, $M_\varepsilon$ can be significantly less than $M$. We use the Ricker map to illustrate this statement.
\begin{example}
\label{ex_Ricker}
Consider the Ricker model given by \eqref{eq:ricker} with $r>2$, so that the positive equilibrium $K=1$ is unstable. We provide lower bounds on $M$ and $M_\varepsilon$ for this model.

Note first that $M$ cannot be less than the magnitude of the slope connecting the point $(K,f(K))=(1,1)$ with the maximum point $(1/r, \exp(r-1)/r)$. This is given by
$$
\frac{e^{r-1} -r}{r-1} \leq M,
$$
where the function in the left-hand side is greater than $r-1$ for $r \geq 2.8$. For instance, $r=5$ leads to the estimates $M>12$.

Now consider $M_\varepsilon$. The derivative $f'(x)=(1-rx)e^{r(1-x)}$ at $x=1$ is $f'(1)=1-r$, so by continuity of the derivatives of \eqref{eq:ricker}, for any $M_\varepsilon>|1-r|=r-1$ there will be some interval $(1-\varepsilon,1+\varepsilon)$ upon which \eqref{eq:M1cond} holds. Again, $r=5$ leads to the estimate $M_\varepsilon>4$.

Let us take $M_\varepsilon=4.5$, then \eqref{eq:M1cond} is satisfied with $\varepsilon =0.6$ 
\begin{equation}
\label{Ricker_cond}
|f(x)-1| \leq  4.5 |x-1|, \quad x \in (0.94, 1.06).
\end{equation}
In fact, the right endpoint of the interval upon which this inequality holds can be chosen to be arbitrarily large, since $|f(x)-1| < 4(x-1)$ for $x>1$. 
\end{example}
\begin{remark}
The ability to choose $M_\varepsilon<M$ will help us to identify where stochastic perturbations may act to stabilise the positive equilibirum $K$.
\end{remark}

The following Borel-Cantelli lemma (see, for example, \cite[Chapter 2.10]{Shir}) will be used in the proof of Lemma \ref{cor:barprob} .

\begin{lemma}
\label{lem:BC}
Let $A_1, \dots, A_n, \dots$ be a sequence of independent events. 
If $\displaystyle \sum_{i=1}^\infty \mathbb P\{A_i\}=\infty$, then $\mathbb P\{A_n \,\,  \text{occurs infinitely often}\}=1.$
\end{lemma}

\begin{lemma}
\label{cor:barprob}
Let $\xi_1, \dots , \xi_n, \dots$ be a sequence of independent identically distributed random variables such that $\mathbb P \left\{\xi_n\in (a,b)\right\}=\tau\in (0, 1)$ for some interval $(a,b)$, $a<b$, and each $n\in\mathbb{N}_0$. Let $n_0$ be an a.s. finite random number. Then for each  $j \in {\mathbb N_0}$, 
\begin{multline*}
\mathbb{P}\left[\text{There exists } \mathcal N=\mathcal N(j)\in (n_0, \infty)\right.\\
\left.\text{ such that (s.t.) } \xi_{\mathcal N}\in (a,b), \, \xi_{\mathcal N+1}\in (a,b), \dots, \xi_{\mathcal N+j}\in (a,b)\right]=1.
\end{multline*}
\end{lemma}

\begin{proof}
Denote
\begin{eqnarray*}
B_1&:=&\{\xi_{1}\in (a,b), \, \xi_{2}\in (a,b), \dots, \xi_{j}\in (a,b)\}, \\
B_2&:=&\{\xi_{j+1}\in (a,b), \xi_{j+2}\in (a,b), \dots, \xi_{2j}\in (a,b)\},\\
&\vdots&\\
B_i&:=&\{\xi_{(i-1)j+1}\in (a,b), \, \xi_{(i-1)j+2}\in (a,b), \dots, \xi_{ij}\in (a,b)\}, \quad \text{for each} \quad i,j\in \mathbb N_0.
\end{eqnarray*}
The events in the sequence $\{B_n\}_{n\in\mathbb{N}_0}$ are mutually independent, since terms of the  sequence $\{\xi_n\}_{n\in\mathbb{N}_0}$ are mutually independent, and each $\xi_i$ appears in one and only one event $\{B_j\}_{n\in\mathbb{N}_0}$. Moreover, we have
\[
\mathbb P\{B_i\}=\tau^j
\]
and therefore
\[
\sum_{i=1}^\infty\mathbb{P}[B_i]=\infty.
\]
The Borel-Cantelli Lemma, Lemma~\ref{lem:BC}, yields
\[
\mathbb P[B_n \,\,  \text{occurs infinitely often}]=1.
\]
Denoting 
\[
\mathcal N(\omega)=\min\left\{ n\ge n_0(\omega): B_n(\omega) \,\, \text{occurs} \right\}
\]
we complete the proof.

\end{proof}

\begin{lemma}
\label{lemma_add1}
Suppose that Assumption~\ref{as:slope} holds, $M_\varepsilon>1$ and $\varepsilon \in (0,K)$ are constants for which \eqref{eq:M1cond} is valid and the unperturbed control parameter satisfies
\[
\alpha\in \left(1-\frac 1{M_\varepsilon}, \, 1\right).
\]
Suppose also that Assumption \ref{as:chi1} holds and further that the perturbation intensity satisfies \begin{equation}
\label{for_cond2}
l<\min\left\{\alpha-\left(1-\frac{1}{M_\varepsilon}\right), \,  \frac{1-\alpha}{\nu}  \right\}.
\end{equation}

Let $\{x_n\}_{n\in\mathbb{N}_0}$ be any solution of equation \eqref{eq:multi} with $x_0>0$. The following is true for all $\omega\in\Omega$:  if $x_j(\omega) \in (K-\varepsilon,K+\varepsilon)$ for some $j \in {\mathbb N_0}$, then $x_n(\omega) \in (K-\varepsilon,K+\varepsilon)$ for any $n \geq j$ and 
\[
\lim_{n \to \infty} x_n(\omega)=K.
\]
\end{lemma}

\begin{proof}
Denote  $\alpha_n(\omega)=\alpha + l \xi_{n+1}(\omega)$ for any $\omega\in\Omega$. We fix and 
notationally suppress the trajectory $\omega$ for the remainder of this proof. From \eqref{for_cond2} and Assumption \ref{as:chi1} we have
\begin{equation}\label{eq:1map}
0<\alpha-l<\alpha_n\le \alpha+ \nu l<1,
\end{equation}
and
\[
0<1-\alpha-\nu l<1-\alpha_n<1-\alpha+l<1.
\]
Moreover,
\begin{equation}\label{eq:MeBound}
(1-\alpha_n) M_\varepsilon<(1-\alpha+l)M_\varepsilon<1.
\end{equation}

Let $x_j \in (K-\varepsilon,K+\varepsilon)$ for some $j \in {\mathbb N_0}$, and suppose that $x_j,x_{j+1}\neq K$. There are two possibilities: either $(x_{j+1}-K)(x_j-K)>0$ or $(x_{j+1}-K)(x_j-K)<0$.

If $(x_{j+1}-K)(x_j-K)>0$ then, as in the proof of Lemma~\ref{lem:PBC}, $|x_{j+1}-K|<|x_j-K|$
and thus $x_{j+1} \in (K-\varepsilon,K+\varepsilon)$.

If instead, $(x_{j+1}-K)(x_j-K)<0$, then we must consider two further sub-cases. Suppose first that $x_j \in (K-\varepsilon, K)$, then $f(x_j)>K$, so that $F_{\alpha_j}(x_j) > K$, and $x_{j+1} \in (K,f(x_j))$. 
Therefore, by \eqref{eq:F}, \eqref{eq:M1cond} and \eqref{eq:MeBound},
\begin{eqnarray*} 
\left| x_{j+1}-K \right| & = &  F_{\alpha_j}(x_j)-K  \\
& = &   (1-\alpha_j) (f(x_j)- K) - \alpha_j (K-x_j)  \\
& \leq &  (1-\alpha_j) M_\varepsilon(K-x_j) - \alpha_j (K-x_j) \\
& \leq &  |x_j-K| - \alpha_j |x_j-K|\\ &\leq&  (1-\alpha + l) \left| x_{j}-K \right|.
\end{eqnarray*}
It follows that $x_{j+1} \in (K-\varepsilon,K+\varepsilon)$. 
Next, when $x_j \in (K,K+\varepsilon)$ we have $f(x_j) < K$, so that $F_{\alpha_j}(x_j) < K$  and $x_{j+1} \in (f(x_j),K)$. Again this yields
\begin{eqnarray*} 
\left| x_{j+1}-K \right| & = & K -  F_{\alpha_j}(x_{j}) \\ & =  &  
 (1-\alpha_j) (K-f(x_j)) - \alpha_j (x_j-K)  \\
& \leq &  (1-\alpha_j) M_\varepsilon(x_j-K) - \alpha_j (x_j-K) \\
& \leq &  \left| x_{j}-K \right| - \alpha_j \left| x_{j}-K \right| \\ & \leq &  (1-\alpha+l) \left| x_{j}-K \right|,
\end{eqnarray*}   
where $0<1-\alpha+l<1$.
Thus $x_{j+1} \in (K-\varepsilon,K+\varepsilon)$, and by induction, all $x_i \in (K-\varepsilon,K+\varepsilon)$, $i \geq j$. We also conclude that the sequence $\{| x_{n}-K |\}_{n\geq j}$ is non-negative 
and monotone non-increasing and therefore has a limit which can only be zero. The result follows.
\end{proof}

\begin{lemma}
\label{lemma_add3}
Suppose that Assumptions~\ref{as:slope}, \ref{as:3} and condition \eqref{eq:M1cond}, with $\varepsilon \in (0,K)$ and $M_\varepsilon \in (1,M)$, hold. Suppose further that Assumption~\ref{as:chi1} is satisfied, and
\begin{equation}
\label{eq_prob}
a:= \alpha - l > 1-\frac{1}{M_\varepsilon}, \quad 1-\frac{1}{M} < b:= \alpha + l \nu <1.
\end{equation}

Let $\{x_n\}_{n\in\mathbb{N}_0}$ be any solution of equation \eqref{eq:multi} with $x_0>0$. Then
\[
\lim_{n \to \infty} x_n = K,\quad a.s.
\]
\end{lemma}

\begin{proof}
By Lemma~\ref{lemma_add1}, it suffices to prove that 
\[
\mathbb{P}\left[\text{There exists }\mathcal N_1<\infty\text{ s.t. }x_{\mathcal N_1} \in (K - \varepsilon, K+ \varepsilon)\right]=1.
\]
Denote as usual $\alpha_n = \alpha + l \xi_{n+1}$ and fix $\varepsilon>0$. Let $N_1$ and $N_2$ be nonrandom positive integers defined by \eqref{def:N1} and \eqref{def:N2} respectively. Our proof is in three parts.
In Part (i) we show that each trajectory $x_n(\omega)$ reaches $[\mu_1, \mu_2]$  in an $\omega$-dependent number of steps and stays there forever. 
In Part (ii) we prove that each trajectory $x_n(\omega)$ then enters $[c, \mu_2]$ in fewer than $N_1$ steps. 
In Part (iii) we verify that each trajectory then enters $[K-\varepsilon, K+\varepsilon]$ in fewer than $N_2$ 
steps. 
\begin{itemize}
\item[Part (i):] Let $c$ be the constant associated with the map $f$ in Assumption \ref{as:3}, and let $\mu_1$ and $\mu_2$ be as defined by \eqref{eq:mu1def} and \eqref{eq:mu2}  in Definition \ref{def:mus}.
 Since $f(x) >x$ for any $x \in [\mu_1,c]$ and $f$ is continuous, we can introduce
\begin{equation}
\label{d1}
d_1 := \min_{x \in [\mu_1, c]} (f(x)-x) >0,
\end{equation}
and
\begin{equation}
\label{def:N1}
N_1 := \left[ \frac{c-\mu_1}{(1-\alpha - \nu l)d_1} \right] +1,
\end{equation}
where $[t]$ is an integer part of $t$.   Next, we denote $a$ as in \eqref{eq_prob} and fix some $a_1$ satisfying
\begin{equation}
\label{def:a}
a_1 \in \left( 1 - \frac{1}{M}, \,\alpha+l\nu \right).
\end{equation}
Denote  
\begin{equation}
\label{gam1}
\gamma = \max\{ b, 1-a\}= \max\{ \alpha + l \nu, 1 - \alpha + l \},
\end{equation}
where $b$ was defined in \eqref{eq_prob} and from which it follows that $\gamma\in (0, 1)$. Additionally denote
\begin{equation}
\label{def:N2}
N_2 := \left[ \left. \ln \left( \max \left\{ \frac{K-c}{\varepsilon}, \frac{\mu_2 - K}{\varepsilon},1 \right\} \right) 
\right/ (-\ln \gamma) \right] + 2.  
\end{equation} 

From \eqref{eq_prob}, we have
\begin{equation}
\label{est:alphan}
\alpha_n \in (a,b) = (\alpha-l, \alpha+l\nu)\subset (\delta, 1-\delta), \quad  1-\alpha_n \geq	 1-\alpha-l\nu,
\end{equation}
where $\delta>0$ can be chosen, for example, to satisfy the inequality
$$
\delta <\min \left\{ a,1-b \right\}.
$$
So the conditions of Lemma~\ref{lemma_add2} are satisfied and we can deduce that for each trajectory $\omega\in\Omega$, there is a finite $n_0(\omega)$ such that $x_{j}(\omega) \in [\mu_1,\mu_2]$ for $j \geq n_0(\omega)$.

Corollary \ref{cor:barprob} implies that,  given any constant integers $N_1$ and $N_2$ defined by \eqref{def:N1} 
and \eqref{def:N2} respectively, the finite random number $n_0$, and the interval 
$\left(\frac {a_1-\alpha}l, \nu\right)$, where $a_1$ is defined by \eqref{def:a},
\begin{multline}\label{def:mathcalN}
\mathbb{P}\left[\text{There exists }\mathcal{N}\in (n_0, \infty)\text{ s.t. }\right.\\ \left.\xi_k\in \left(\frac {a_1-\alpha}l, \nu\right),\text{ for all } k=\mathcal N+1,  \dots, 
\mathcal N+N_1+N_2,
\right]=1.
\end{multline}
Since $a_1<\alpha+l\nu$, we have
\[
\left(\frac {a_1-\alpha}l, \nu\right)\cap (-1, \nu)\neq \emptyset, 
\]
so that
\[
\mathbb P\left\{\xi_k\in \left(\frac 
{a_1-\alpha}l, \nu\right)\right\}>0,\quad \text{for all}\quad k\in\mathbb{N}_0.
\]
Also, on a given trajectory $\omega\in\Omega$, 
\begin{equation}
\label{cond:alphak}
\xi_k(\omega)\in \left(\frac {a_1-\alpha}l, \nu\right)\quad \Rightarrow\quad\alpha_k(\omega)\in (a_1, \alpha+l\nu).
\end{equation}

\item[Part (ii):] From Assumption \ref{as:slope} and \eqref{est:alphan}, note that for any fixed trajectory $\omega\in\Omega$, as long as $x_n(\omega) \in [\mu_1,c]$, we have
\begin{eqnarray*}
x_{n+1}(\omega) & = & \alpha_n(\omega) x_n(\omega) +(1-\alpha_n(\omega))f(x_n(\omega))\\
&=&x_n(\omega)+(1-\alpha_n(\omega))(f(x_n(\omega))-x_n(\omega)) \\ 
& \geq & x_n(\omega) + (1-\alpha_n(\omega)) d_1\\
& \geq &x_n(\omega)+ (1-\alpha -l\nu) d_1,
\end{eqnarray*}
and hence at least one of $x_{n+1}(\omega), \dots, x_{n+N_1}(\omega)$ is in $[c,\mu_2]$.

\item[Part (iii):] For any fixed trajectory $\omega\in\Omega$, if $x_n(\omega) \in [c,\mu_2]$ and $N_2$ successive terms of the subsequence
$\{\alpha_k(\omega)\}_{k=n}^{\infty}$ satisfy 
\begin{equation}
\label{big_alpha}
\alpha_{k-1}(\omega)=\alpha+l \xi_{k}(\omega) \in (a_1, \alpha + l\nu), \quad k=n+1, n+2, \dots, n+N_2,
\end{equation}
we have $x_{n+N_2+1} \in (K-\varepsilon, K+\varepsilon)$.  For the proof we assume that at least one of $K-c$, $\mu_2-K$ is not less than $\varepsilon$, otherwise $x_n(\omega) \in [c,\mu_2]$
is already in $(K-\varepsilon, K+\varepsilon)$.  

Choose $\omega$ to be any trajectory in the a.s. event described by \eqref{def:mathcalN}, and suppose that $x_n(\omega) \in [c, K-\varepsilon] \cup [K+\varepsilon, \mu_2]$. It follows from \eqref{def:mathcalN} and \eqref{cond:alphak} that on this trajectory $\alpha_{k-1}(\omega)$ satisfies \eqref{big_alpha}, for $k=n+1, \dots, 
n+N_2$. 

Applying Lemma~\ref{lem:PBC} with $b=\alpha+\nu l$, $a_1$ as chosen in \eqref{def:a} in place of $a$, 
and $\gamma$ defined in \eqref{gam1}, we arrive at 
$$
\left| x_{j+1}(\omega)-K  \right|   \leq   \gamma \left| x_{j}(\omega)-K \right|,
$$
for $j=n+1, n+1, \dots, n+N_2$. If $x_n(\omega) \in [c,K-\varepsilon]$, this  relation implies 

\begin{eqnarray*}
\left| x_{n+N_2}(\omega) -K \right| &\leq& {\gamma}^{N_2} \left| x_n(\omega)-K \right|\\ 
&\leq& {\gamma}^{-\log_{\gamma}((K-c)/\varepsilon)} \left| x_n(\omega)-K \right| \\
&<& \frac{\varepsilon}{K-c} \left| x_n(\omega)-K \right| \leq \varepsilon.
\end{eqnarray*}
Similarly, for $x_n(\omega) \in [K+\varepsilon,\mu_2]$,
\begin{eqnarray*}
\left| x_{n+N_2}(\omega) -K \right|  &\leq& {\gamma}^{N_2} \left| x_n(\omega)-K \right| \\
&\leq& {\gamma}^{-\log_{\gamma}((\mu_2-K)/\varepsilon)} \left| x_n(\omega)-K \right| \\
&<& \frac{\varepsilon}{\mu_2-K} \left| x_n(\omega)-K \right| \leq \varepsilon.
\end{eqnarray*}
\end{itemize}

Now we bring together all parts of the proof. In Part (i), we verified that there exists a finite  random 
number 
$n_0$ such that, for any $\omega\in\Omega$, $x_n(\omega) \in [\mu_1,\mu_2]$ for $n \geq n_0(\omega)$. For  $N_1$ and $N_2$ and $a_1$,  
defined in \eqref{def:N1}, \eqref{def:N2} and \eqref{def:a}, respectively, 
\eqref{def:mathcalN} holds for some random number $\mathcal N$. Then, in particular, $x_n\in [\mu_1,\mu_2]$ for all $n \ge\mathcal N(\omega)>n_0(\omega)$.

In Part (ii), we proved that there exists  $k(\omega)\in [0,  N_1]$, such that for any $\omega\in\Omega$
\[
x_{\mathcal N(\omega)+k(\omega)}\in (c, \mu_2).
\]
Finally, in Part (iii), we showed that 
\[
\mathbb{P}\left[\omega\in \Omega: |x_{\mathcal{N}_1(\omega)}(\omega)-K|\leq \varepsilon \quad \text{for} \quad \mathcal{N}_1:=\mathcal{N}+N_1+N_2 \right]=1.
\]
An application of Lemma~\ref{lemma_add1} concludes the proof.
\end{proof}

We recall that by Corollary \ref{new_add1}, the positive equilibrium of the 
unperturbed PBC equation with constant $\alpha$ is globally asymptotically stable if we choose $\alpha>1-1/M$. 
The next theorem presents conditions under which the introduction of a stochastic perturbation of $\alpha$ 
has the effect of 
a.s. stabilising the positive equilibrium. In both cases, the presence of noise has the effect of ensuring that on an event of probability one, solutions (regardless of a positive initial value) will eventually enter the domain of local stability predicted by Lemma \ref{lemma_add1}. In this sense, we are showing that pathwise local stability, together with an appropriate noise perturbation, imply a.s. global stability.

In the first part, we require the support of the stochastic perturbation to be symmetric ($\nu=1$). A.s. global stability of the equilibrium of the stochastic PBC can be achieved by an appropriate choice of noise intensity $l$ if $\alpha\leq 1-1/M$ as long as $\alpha$ is closer to $1-1/M$ than to $1-1/M_\varepsilon$. Hence the parameter range corresponding to known global asymptotic stability is extended by an appropriate stochastic perturbation.

In the second part, we show that the a.s. stability region can be extended further, to $\alpha > 1/M_\varepsilon$, by allowing the support of the perturbations to extend to the right ($\nu>1$), so that, with probability one, values of the sequence $\{\alpha_n\}_{n\in\mathbb{N}_0}$ exceed $1-1/M$ sufficiently often on an a.s. event.

\begin{theorem}
\label{multi_local}
Suppose that Assumptions~\ref{as:slope}, \ref{as:3} hold, and
there is an $\varepsilon>0$ and 
$M>M_\varepsilon>1$ such that \eqref{eq:M1cond} is satisfied.
Suppose further that one of the two following conditions holds:
\begin{enumerate}
\item Assumption~\ref{as:chi1} holds with $\nu=1$,
\begin{equation}\label{eq:al1}
\alpha \in\left (1 - \frac{1}{2 M_\varepsilon}-\frac{1}{2M},1\right), 
\end{equation}
and 
\begin{equation}
\label{eq:noi_M1}
l \in \left( \max\left\{1-\frac{1}{M} -\alpha,0\right\},\, \min\left\{\alpha-\left(1-\frac{1}{M_\varepsilon}\right),\,1-\alpha\right\}\right);
\end{equation}
\item Assumption~\ref{as:chi1} holds with 
$\nu>1$, and
\begin{equation}\label{alpha_cond}
\alpha \in \left(1 - \frac{1}{M_\varepsilon},1\right)
\end{equation} 
is such that
\begin{equation}
\label{nu_cond}
1-\frac{1}{M} -\alpha < \nu \left( \alpha - \left(1 - \frac{1}{M_\varepsilon}\right)  \right)
\end{equation}
and
\begin{equation}
\label{eq:n_M1}
l \in \left( \frac{1}{\nu} \max \left\{  1-\frac{1}{M} -\alpha, 0 \right\}, \min\left\{ \alpha - \left(1 -\frac{1}{M_\varepsilon}\right), \frac{1-\alpha}{\nu} \right\} \right).
\end{equation}
\end{enumerate}

Let $\{x_n\}_{n\in\mathbb{N}_0}$ be any solution of equation \eqref{eq:multi} with $x_0>0$. Then
\[
 \lim_{n\to\infty}x_n=K,\quad \text{a.s.}
\]
\end{theorem}
\begin{proof}
We treat each part in turn, showing that the conditions of Lemma \ref{lemma_add3} are satisfied.
\begin{enumerate}
\item Let Assumption~\ref{as:chi1} hold with $\nu=1$. By \eqref{eq:al1}
$$ 
1-\frac{1}{M} -\alpha < \alpha-\left(1-\frac{1}{M_\varepsilon}\right), 
$$
and so the interval for $l$ in \eqref{eq:noi_M1} is non-empty. Thus, applying \eqref{eq:noi_M1},
\[
\alpha+l<1, \quad \alpha+l>1-\frac{1}{M}
\]
and, additionally by \eqref{alpha_cond}
\[
\alpha-l>2\alpha-\left(1-\frac{1}{M}\right)>1-\frac{1}{M}>1-\frac 1{M_\varepsilon}.
\]
The result then follows from Lemma \ref{lemma_add3}.

\item Let Assumption~\ref{as:chi1} hold with $\nu>1$. 
It is clear that $(1-1/M-\alpha)/\nu<(1-\alpha)/\nu$, and this with 
\eqref{nu_cond} ensures that the interval for $l$ in \eqref{eq:n_M1}
is non-empty. It remains to see that \eqref{eq:n_M1} implies
$$ l\nu < 1- \alpha, \quad \alpha - l>1 -\frac{1}{M_\varepsilon}, $$
and
\[
\alpha+l \nu> 1 - \frac{1}{M}.
\]
The result then follows from Lemma \ref{lemma_add3}.
\end{enumerate}
\end{proof}


\section{Additive Noise}
\label{sec:4}


Consider the case where the PBC equation is perturbed externally, 
and therefore includes an additive stochastic term
\begin{equation}
\label{eq:add}
x_{n+1}= \max \left\{ f(x_n)-\alpha (f(x_n)-x_n) +  l\xi_{n+1}, 0 \right\}, \quad x_0>0, \quad n \in {\mathbb N_0}.
\end{equation}
Since the perturbation intensity $l$ is fixed, a.s. asymptotic convergence to the positive point equilibrium $K$ is impossible. 
However, we can show that if $K$ is a stable equilibrium of the unperturbed prediction-based controlled 
equation, then solutions of 
the perturbed PBC equation eventually enter and remain within a neighbourhood of $K$, where the size of that neighbourhood depends on $l$. In this section we suppose that the support of the perturbation is symmetric ($\nu=1$).

\begin{lemma}
\label{lem:add_noise}
Suppose that Assumptions~\ref{as:slope},\ref{as:3} and \ref{as:chi1} (with $\nu=1$) hold,
and that
\[
\alpha \in \left(1-\frac{1}{M},1\right).
\] 
Let $\{x_n\}_{n\in\mathbb{N}_0}$ be any solution of \eqref{eq:add} with $x_0>0$. If 
\begin{equation}
\label{eq:l_cond}
l< (1-\gamma)(K-c),
\end{equation}
where $\gamma$ is defined by \eqref{def_gamma} in the statement of Corollary \ref{new_add1}, then the following statement holds: for any $\varepsilon>0$, 
\begin{equation}\label{eq:bounds}
\mathbb{P}\left[\text{There exists } \mathcal{N}_0<\infty\text{ s.t. }x_n\in\left(K-\frac{l}{1-\gamma}-\varepsilon, K+\frac{l}{1-\gamma}+ \varepsilon\right),\,n\geq\mathcal{N}_0\right]=1.
\end{equation}
\end{lemma}
\begin{proof}
Our proof has three parts.
\begin{itemize}
\item[Part (i):] We show that for all $\omega\in\Omega$, if  $x_n(\omega) \in [c,2K-c]$, then all $x_j(\omega) \in (c,2K-c)$, for $j \geq n$. 

Since $\alpha \in (1-1/M,1)$, then by \eqref{contraction} in Corollary~\ref{new_add1}, Assumptions 
\ref{as:3} and \ref{as:chi1} ($\nu=1$), and \eqref{eq:l_cond}, if $x_n(\omega)\in [c, 2K-c]$ and $F_\alpha$ is defined by \eqref{eq:F},
\begin{eqnarray*}
\left| x_{n+1}(\omega)-K \right| & = & \left| F_{\alpha} (x_n(\omega)) +  l\xi_{n+1} - K \right| \\
&\leq& \left| F_{\alpha} (x_n(\omega))  - K \right| +  \left|  l\xi_{n+1} \right| \\
& \leq & \gamma \left| x_{n}(\omega)-K \right| + l\\ &<& \gamma (K-c) + (1-\gamma) (K-c) = K-c, 
\end{eqnarray*}
thus $x_{n+1}(\omega) \in (c,2K-c)$. By induction, $x_j(\omega) \in [c,2K-c]$ for any $j \geq n$.\\

\item[Part (ii):]
We prove that for any $x_0 \notin [c,2K-c]$ 
\[
\mathbb{P}\left[\text{There exists }\mathcal N_1<\infty\text{ s.t. }x_{\mathcal{N}_1} \in [c,2K-c]\right]=1.
\]
If  $x_0 \in [0,c]$ then the following is true for all $\omega\in\Omega$: either $x_n(\omega) \in [0,c]$ for all $n\in \mathbb N_0$, or $x_m(\omega) >c$ for some unknown $m \in \mathbb N_0$. 
If $x_n(\omega) \in [0,c]$ for all $n\in \mathbb N_0$,  denote
\begin{equation}
\label{def:r}
r=\left[ \frac{2 c}{l} \right] +1.
\end{equation}
By Lemma~\ref{cor:barprob}, 
\begin{equation*}
\mathbb{P}\left[\text{There exists }\mathcal N_2<\infty\text{ s.t. }\xi_{j} \in \left( \frac{1}{2}, 1\right)\text{ when }j=\mathcal N_2+1,  \mathcal N_2+2, \dots, \mathcal N_2+r\right]=1
\end{equation*}
Note that if $\omega\in\Omega$ is such that for some $n\in\mathbb{N}_0$, $x_n(\omega) \in (0, c]$ and $\displaystyle \xi_{n+1}(\omega) \in \left( \frac{1}{2}, 1\right)$,
we have
\begin{equation*}
x_{n+1}(\omega)  =  F_{\alpha}(x_n(\omega))+l\xi_{n+1}(\omega) \geq   x_n(\omega) + \frac{l}{2}.
\end{equation*}
Therefore 
$$ 
x_{\mathcal N_2+1}(\omega) \geq x_{\mathcal N_2}(\omega)+\frac{l}{2}, ~~\dots~, x_{\mathcal N_2+r}(\omega)\geq x_{\mathcal N_2}(\omega)+r\frac{l}{2} > c,$$
which makes the first case, $x_n(\omega) \in [0,c]$ for all $n\in \mathbb N_0$, impossible.

If  $x_0> 2K-c$ then, by Assumption~\ref{as:3}, $F_{\alpha} (x_0) < F_{\alpha} (2K-c)$, so, by our choice of $l$,
\begin{eqnarray*}
x_1(\omega) &<& F_{\alpha} (2K-c) + l \xi_1(\omega)\\ 
&<& K+l \\
 &<& K + \alpha(K-c)\\ &<& K+K-c\\&=&2K-c.
\end{eqnarray*}
Together with Part (i), this allows us to conclude that
\begin{equation}\label{eq:n1rand}
\mathbb{P}\left[\text{There exists } \mathcal{N}_1<\infty\text{ s.t. }x_{\mathcal{N}_1} \in [c,2K-c]\right]=1.
\end{equation}


\item[Part (iii):] We may now proceed to prove \eqref{eq:bounds}. By \eqref{contraction} in the statement of Corollary~\ref{new_add1}, 
\[
|F_{\alpha}(x)-K | < \gamma |x-K|,\quad x\in[c,2K-c].
\]
By Parts (i) and (ii) we only need to consider trajectories $\omega$ belonging to the a.s. event referred to in \eqref{eq:n1rand}. Introducing the nonnegative sequence $\{y_n(\omega)\}_{n\geq\mathcal{N}_1(\omega)}$ with $y_n(\omega) := \left| x_{n}(\omega)-K \right|$, we notice that  
\begin{equation}
\label{eq:flower}
0\leq y_{n+1}(\omega) < \gamma y_n(\omega) + l.
\end{equation} 
There are two possibilities.
\begin{enumerate}
\item If $\displaystyle y_n(\omega)< \frac{l}{1-\gamma}$ then
$$
y_{n+1}(\omega) < \gamma y_n(\omega) + l < \frac{ \gamma l + (1-\gamma) l}{1-\gamma} = \frac{l}{1-\gamma},
$$
and therefore all successive terms will be on the interval $\left[0,\frac{l}{1-\gamma}\right)$.  

\item If $\displaystyle y_n(\omega) \geq  \frac{l}{1-\gamma}$ then 
$$
y_{n+1}(\omega) < \gamma y_n(\omega) + l  \leq \gamma y_n(\omega) + (1-\gamma) y_n(\omega) = y_n(\omega),
$$
and thus $\{y_n(\omega)\}_{n\in\mathcal{N}_1(\omega)}$ is a positive decreasing sequence for as long as each $y_n \geq l/(1-\gamma)$. 
\end{enumerate}
Hence, either $\{y_n(\omega)\}_{n\in\mathcal{N}_1(\omega)}$ has a limit
$\displaystyle \lim_{n \to \infty} y_n(\omega) = A(\omega)$ satisfying, by \eqref{eq:flower},
$$A(\omega) \leq \frac{l}{1-\gamma},$$
or it eventually drops below $l/(1-\gamma)$. Therefore, for any $\varepsilon>0$, 
\[
\mathbb{P}\left[\text{There exists } \mathcal{N}_0<\infty\text{ s.t. }y_n \in \left(0,\frac{l}{1-\gamma}+\varepsilon\right),\, n \geq \mathcal N_0\right]=1,
\]
which immediately implies \eqref{eq:bounds}, and the statement of the lemma.
\end{itemize} 
\end{proof}

Finally we show that it follows from Lemma \ref{lem:add_noise} that the neighbourhood of $K$ into which solutions eventually settle can be made arbitrarily small by placing an additional constraint on the noise intensity $l$.
\begin{theorem}
Suppose that Assumptions~\ref{as:slope}, \ref{as:3} and \ref{as:chi1} (with $\nu=1$) hold,
and that
\[
\alpha \in \left(1-\frac{1}{M},1\right).
\] 
Let $\{x_n\}_{n\in\mathbb{N}_0}$ be any solution of equation \eqref{eq:add} with $x_0>0$. For any $\varepsilon_1>0$, there exists $l$ satisfying \eqref{eq:l_cond} such that 
\begin{equation}
\label{eq:bound1}
\mathbb{P}\left[\text{There exists } \mathcal{N}<\infty\text{ s.t. }x_n \in \left[ \max\left\{ K-\varepsilon_1, 0 \right\}, K+\varepsilon_1 \right],\, n\geq\mathcal{N}\right]=1.
\end{equation}
\end{theorem}

\begin{proof}
Let us choose in the statement of Lemma~\ref{lem:add_noise},
$$\varepsilon \leq \frac{\varepsilon_1}{2}, \quad l \leq  \min\left\{ \varepsilon(1-\gamma), (1-\gamma)(K-c) \right\},$$
then a reference to \eqref{eq:bounds} in Lemma~\ref{lem:add_noise} completes the proof. 
\end{proof}

\section{Numerical Examples and Discussion}

Our numerical experiments are mostly concerned with the stabilising effect of the multiplicative noise.
First, let us illustrate stabilisation of the chaotic Ricker model using PBC with multiplicative noise.

\begin{example}
\label{ex_Ricker_1}
Consider the chaotic Ricker map \eqref{eq:ricker} with $r=5$.
As mentioned in Example~\ref{ex_Ricker}, inequality \eqref{eq:Mcond} is satisfied with 
$M>12$, while for $M_1=4.5$ and $\varepsilon =0.6$, \eqref{Ricker_cond} holds.
As $M \approx 12.8624$, we can take $M=12.87$ in further computations, \eqref{eq:Mcond} will be satisfied.
Thus, according to Theorem~\ref{multi_local}, we should choose $\alpha$ such that
$$
\alpha > 1 - \frac{1}{M_\varepsilon} \approx 0.778, \quad \alpha -l >1 - \frac{1}{M_\varepsilon}$$,  $$\alpha+ \nu l > 1 - \frac{1}{M} \approx 0.9223,
\quad \alpha+ \nu l <1.
$$
Let us take $\alpha =0.8$, $l= 0.02$, $\nu=1$. Then, $\alpha+l\xi_n > 1 - 1/M_\varepsilon$, 
Fig.~\ref{figure1} shows fast convergence of solutions to the equilibrium $K=1$.

\begin{figure}[ht]
\centering
\includegraphics[height=.28\textheight]{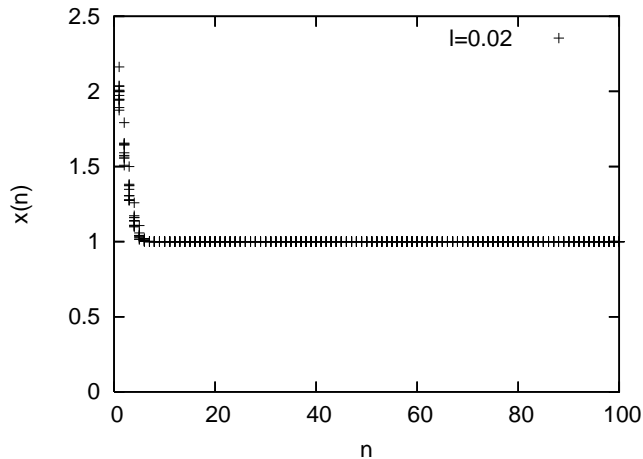}
\caption{Ten runs of the difference equation with $f$ as in (\protect{\ref{eq:ricker}}) with $r=5$ and
multiplicative stochastic perturbations with $\alpha =0.8$, $l=0.02$, $x_0=0.3$, $\nu=1$ and uniformly distributed noise.}
\label{figure1}
\end{figure}

Next, let us take $\alpha<1 - 1/M_\varepsilon$, for example, $\alpha=0.5$.
Fig.~\ref{figure2} illustrates the dynamics of the Ricker equation with deterministic PBC ($l=0$), 
and the multiplicative uniformly distributed noise with the growing perturbation amplitudes $l=0.05, 0.1, 0.2$.

\begin{figure}[ht]
\centering
\includegraphics[height=.23\textheight]{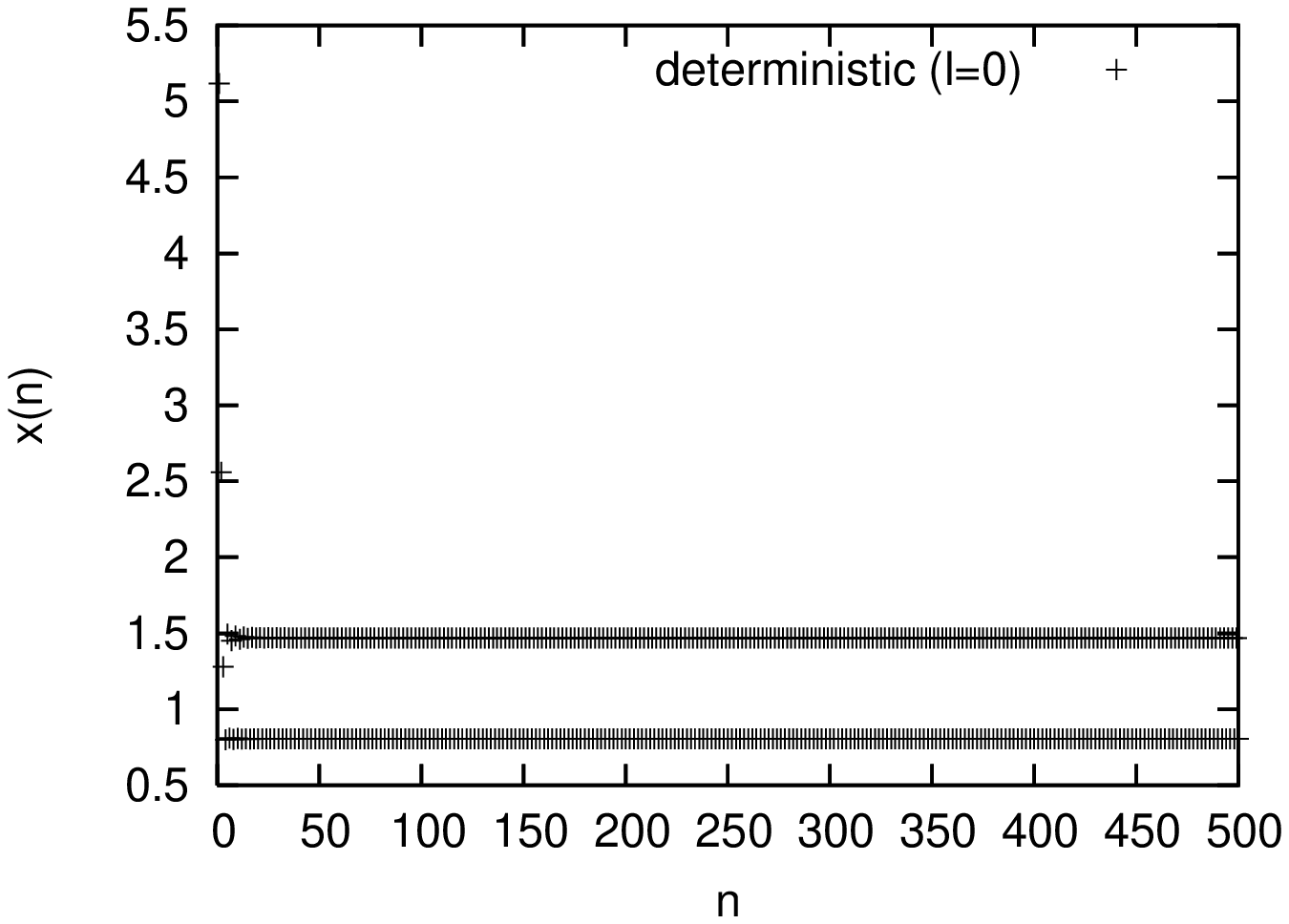}
\hspace{5mm}
\includegraphics[height=.23\textheight]{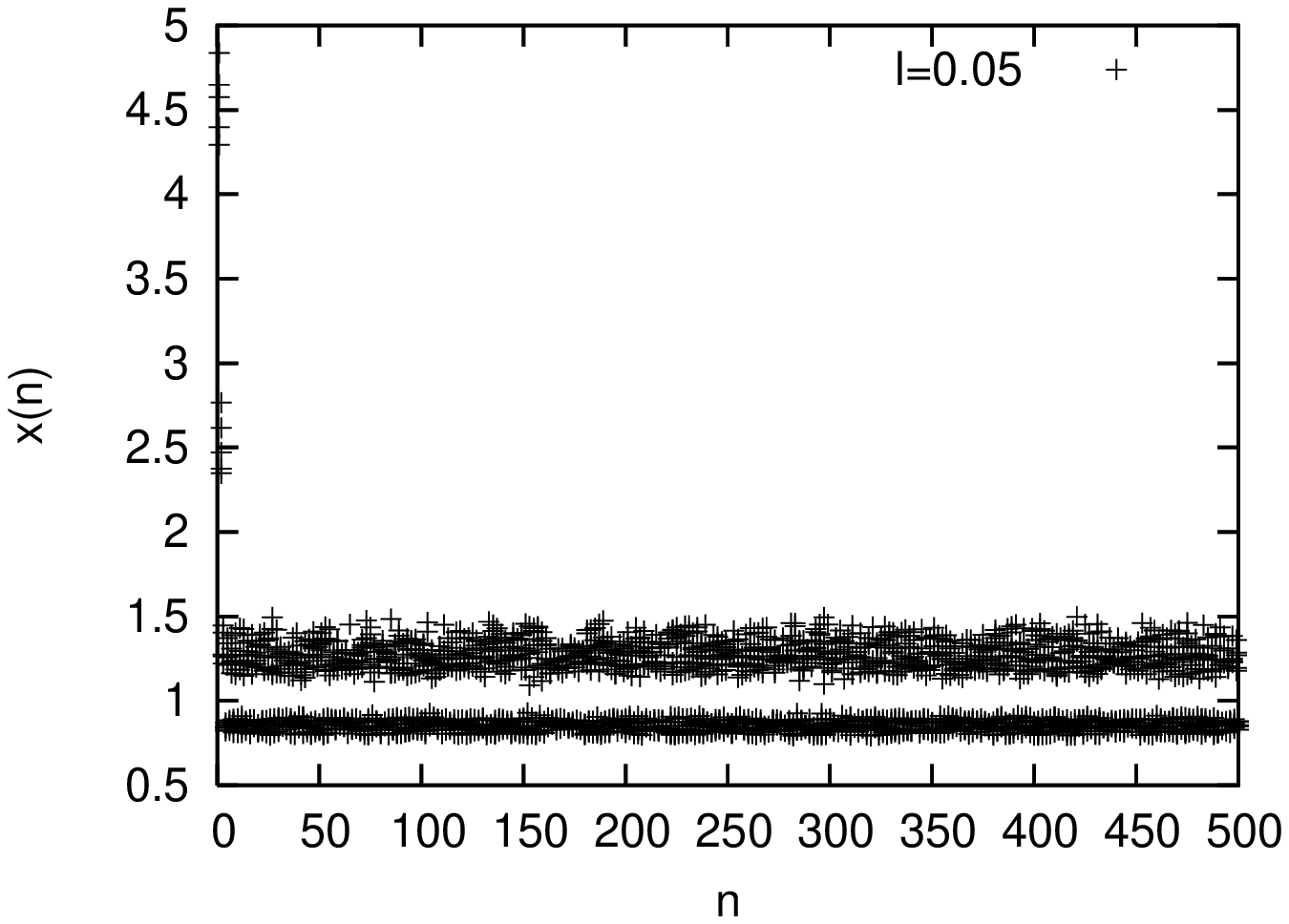}
\hspace{5mm}
\includegraphics[height=.23\textheight]{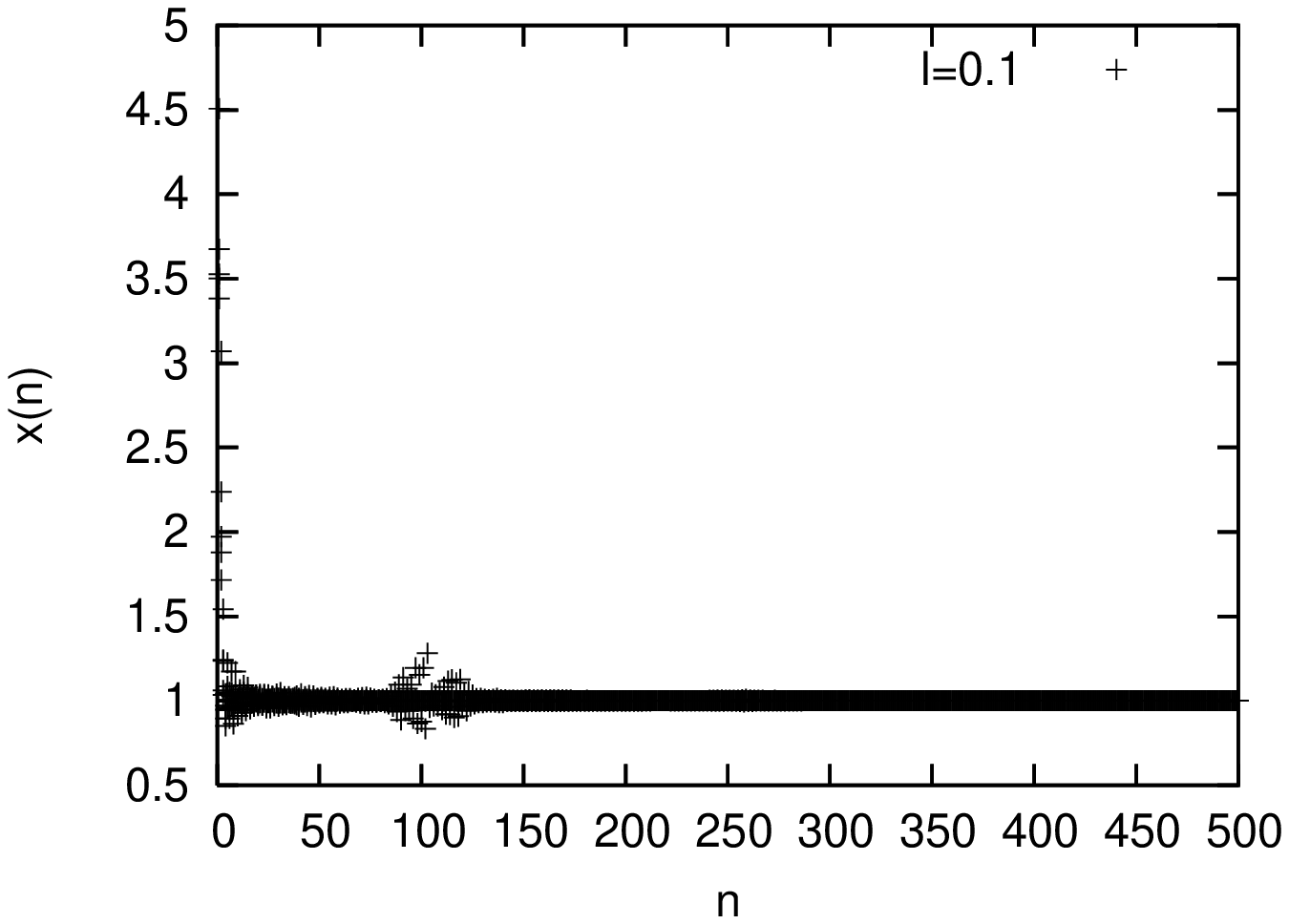}
\hspace{5mm}
\includegraphics[height=.23\textheight]{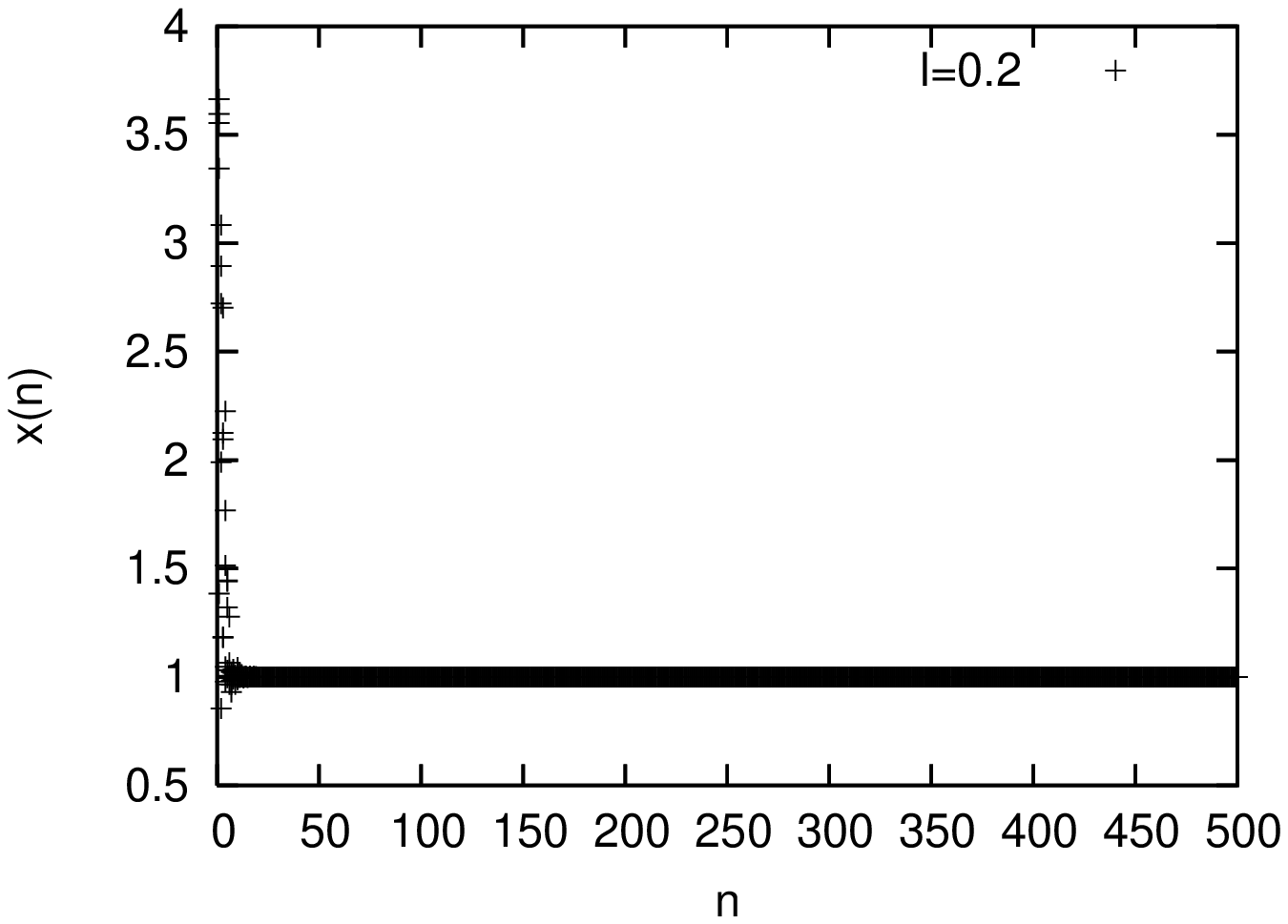}
\caption{Five runs of the difference equation with $f$ as in (\protect{\ref{eq:ricker}}) with $r=5$ and
multiplicative stochastic perturbations, where $\alpha =0.5$, $x_0=0.3$, $\nu=1$, and $l=0, 0.05, 0.1, 0.2$ (from left to right),
and uniformly distributed noise.}
\label{figure2}
\end{figure}

Finally, let us fix $\alpha=5$, $l=0.05$ and increase $\nu$. The distribution function of $\xi$ is chosen to be 
$\displaystyle e^{\ln(\nu+1) \ln(2\zeta)/\ln 2} - 1 $, where $\zeta$ is uniformly distributed on $(0,1)$. 
As $\zeta<0.5$ leads to $\xi<0$, half 
of the perturbations are negative. We can observe the stabilising effect of larger $\nu$ in 
Fig.~\ref{figure3}.

\begin{figure}[ht]
\centering
\includegraphics[height=.165\textheight]{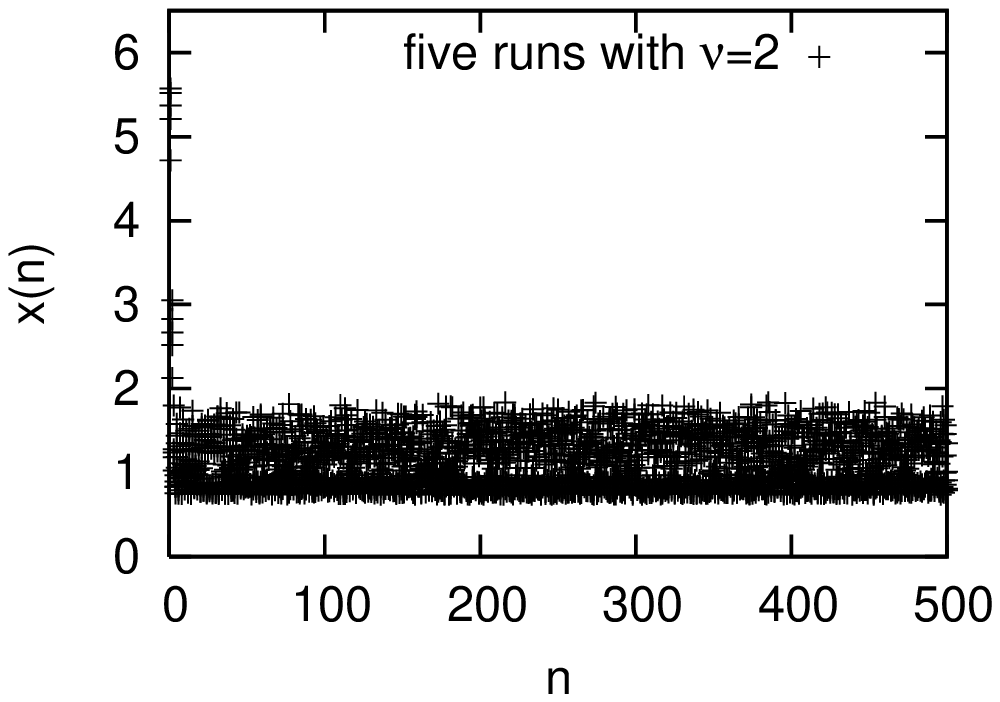}
\hspace{2mm}
\includegraphics[height=.165\textheight]{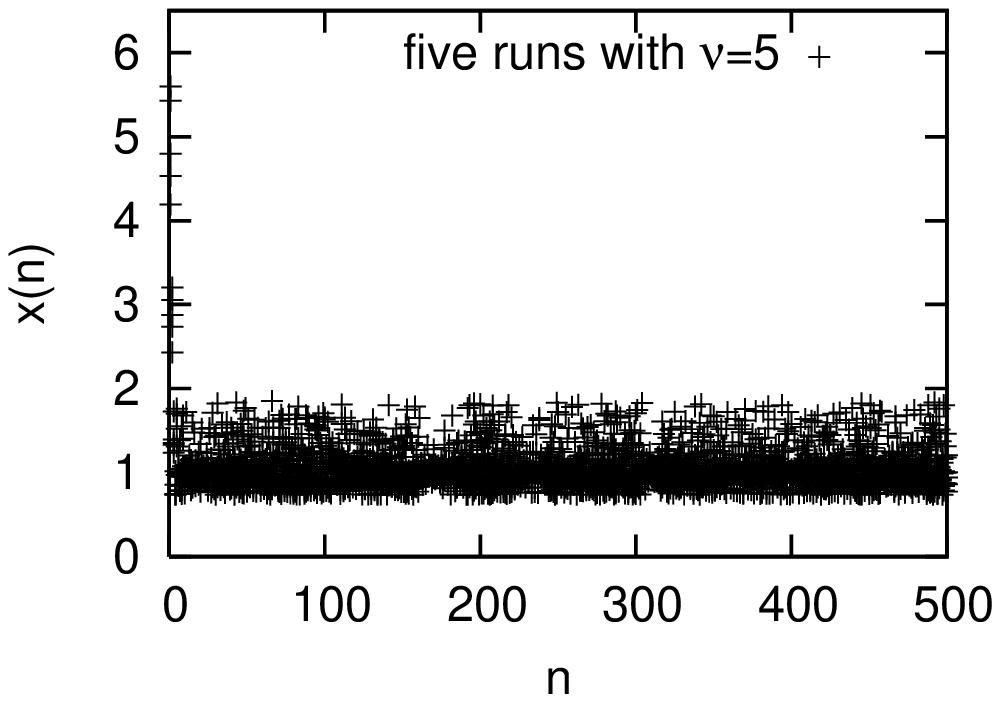}
\hspace{2mm}
\includegraphics[height=.165\textheight]{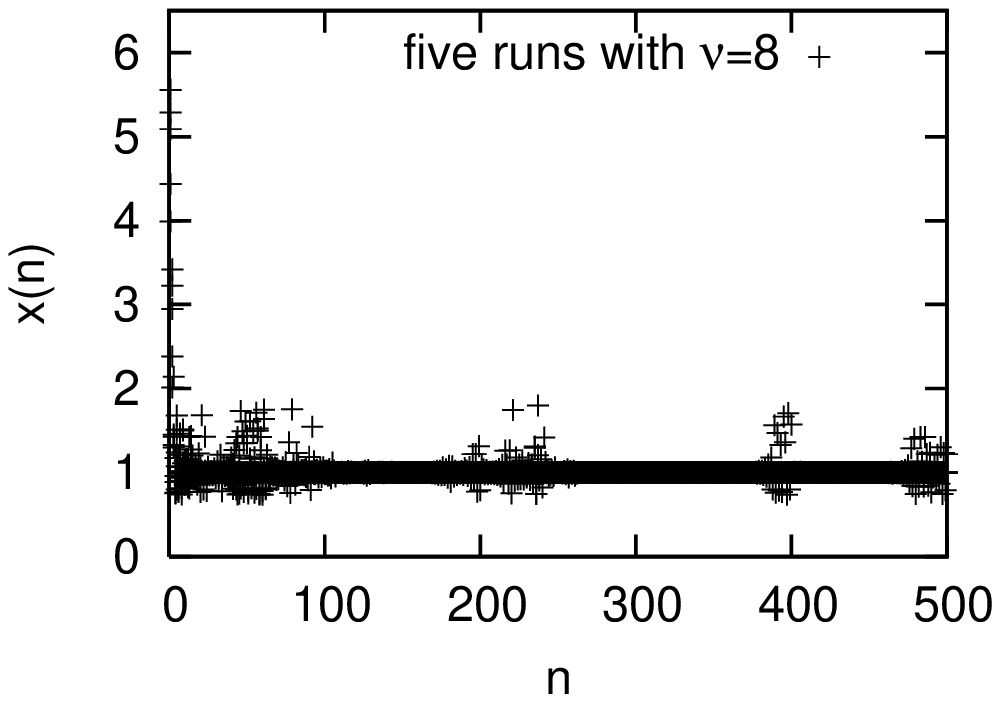}
\caption{Five runs of the difference equation with $f$ as in (\protect{\ref{eq:ricker}}) with $r=5$ and
multiplicative stochastic perturbations, where $\alpha =0.5$, $x_0=0.3$, $l=0.05$ and (left) $\nu=2$, (middle) $\nu=5$, (right) $\nu=8$.}
\label{figure3}
\end{figure}

%
\end{example}

Next, we illustrate the stabilising effect of a multiplicative stochastic perturbation in non-controlled 
models.

\begin{example}
Let us consider the example of Singer \cite[p.~266]{Singer} where a map has a locally stable
equilibrium, together with an attractive cycle.

Denote
$$ F(x)=7.86 x - 23.31 x^2+ 28.75x^3- 13.30x^4, \quad a:=F(0.99) \approx 0.055438317 $$
Consider the function
\begin{equation}
\label{eq:S3.1}
f(x) = \left\{ \begin{array}{ll} F(x),& \mbox{ if~~ } x \in [0,0.99], \\ \frac{F(0.99)}{x+0.01},
& \mbox{ if~~ } x \in  (0.99, \infty). \end{array}
\right.
\end{equation}
Thus
\begin{equation}
f(x) = \left\{ \begin{array}{ll} 7.86 x - 23.31 x^2+ 28.75x^3- 13.30x^4, & \mbox{ if~~ }  x \in [0,0.99], \\
\frac{100 F(0.99)}{100x+1}, &  \mbox{ if~~ } x \in  (0.99, \infty). \end{array} \right.
\label{Singfunc}
\end{equation}
Following \cite{Singer}, we notice that $F$ has a locally stable fixed point
$K \approx 0.7263986$
together with a locally stable period two orbit
$(\theta_1, \theta_2) \approx  (0.3217591,0.9309168)$.
Here  $c  \approx 0.3239799$.
In Fig.~\ref{figure5a}, we illustrate the function $f(x)$ introduced in \eqref{Singfunc} on the 
segment $[0.1.5]$, together with 
a two-cycle which is locally stable in the absence of perturbations.

\begin{figure}[ht]  
\centering
\includegraphics[height=.25\textheight]{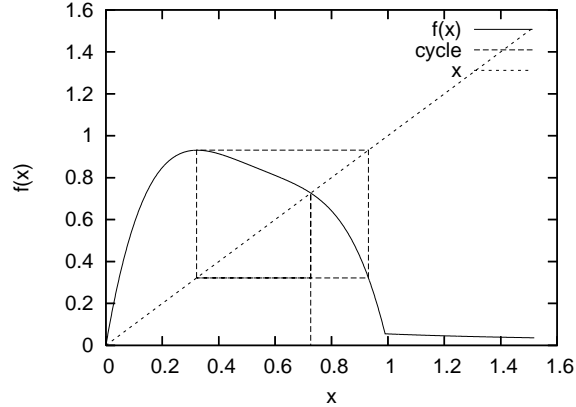}
\caption{
The map $f(x)$ introduced in (\protect{\ref{Singfunc}}) has an equilibrium $K \approx 0.7263986$
and a locally stable two-cycle $(0.3217591,0.9309168)$.}
\label{figure5a}
\end{figure}

In the following 10 numerical simulations, the initial point $x_0=0.3217$ is chosen very close to
the lower value $0.3217591\dots$ of the stable 2-cycle. Here the amplitude of uniformly distributed in $[-l,l]$
perturbations is $l=0.02$, $\xi$ is uniformly distributed in $[-1,1]$ 
\begin{equation}
x_{n+1}=(1+ l \xi)f(x_n). 
\label{pert_1}
\end{equation}
We observe that the stochastic perturbation can make
a locally (though not globally) asymptotically stable equilibrium, globally asymptotically stable. An important condition of this global
stability is that there is a neighbourhood of the equilibrium which is invariant for any perturbations.
On the other hand, the occasional perturbations amplitude should be large enough to leave the stable 2-orbit.
If we increase the amplitude to $l=0.03$, the process of attraction of solution to the locally stable equilibrium
is faster, see Fig.~\ref{figure5}, right.

\begin{figure}[ht]
\centering
\includegraphics[height=.24\textheight]{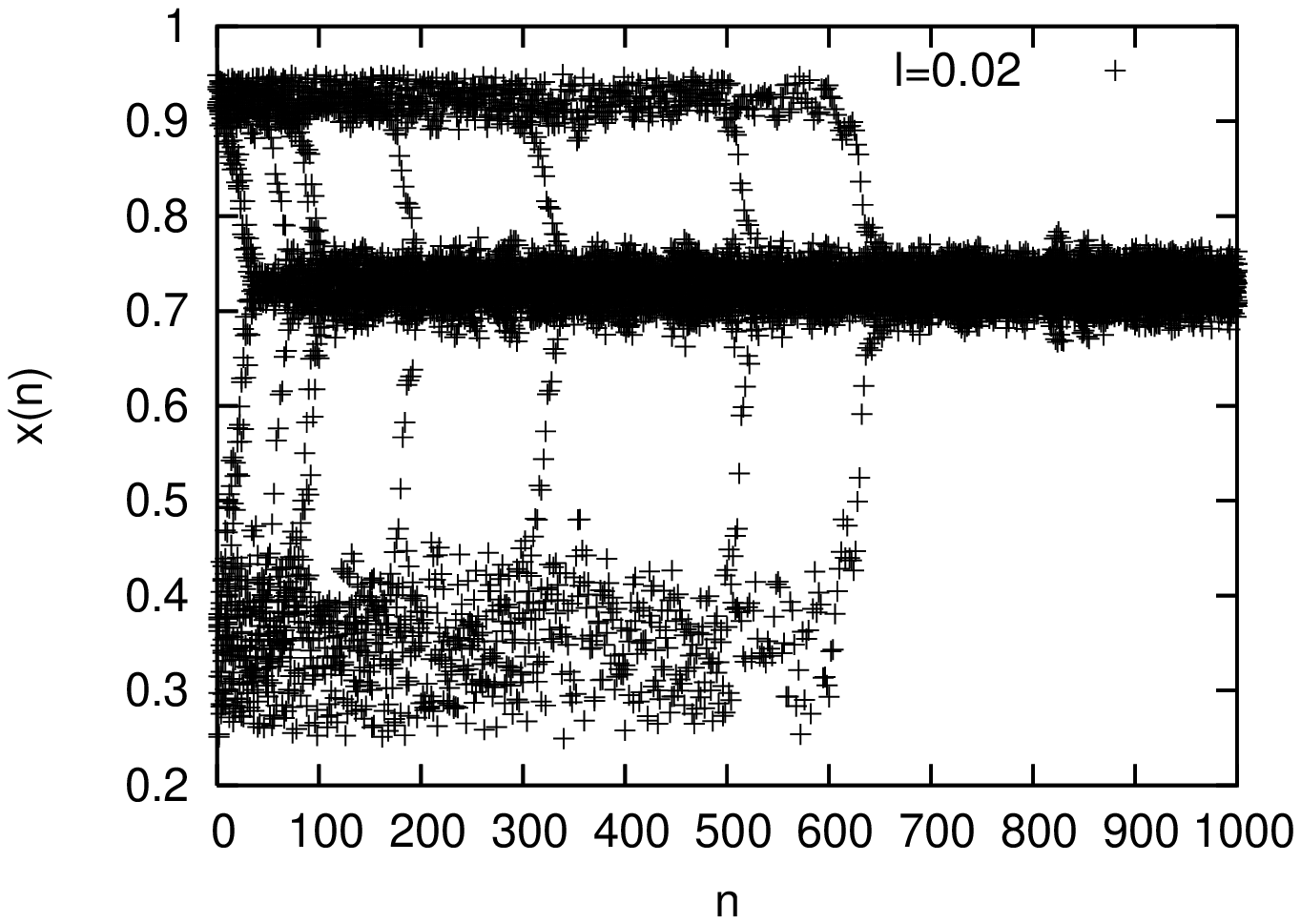}
\hspace{8mm}
\includegraphics[height=.24\textheight]{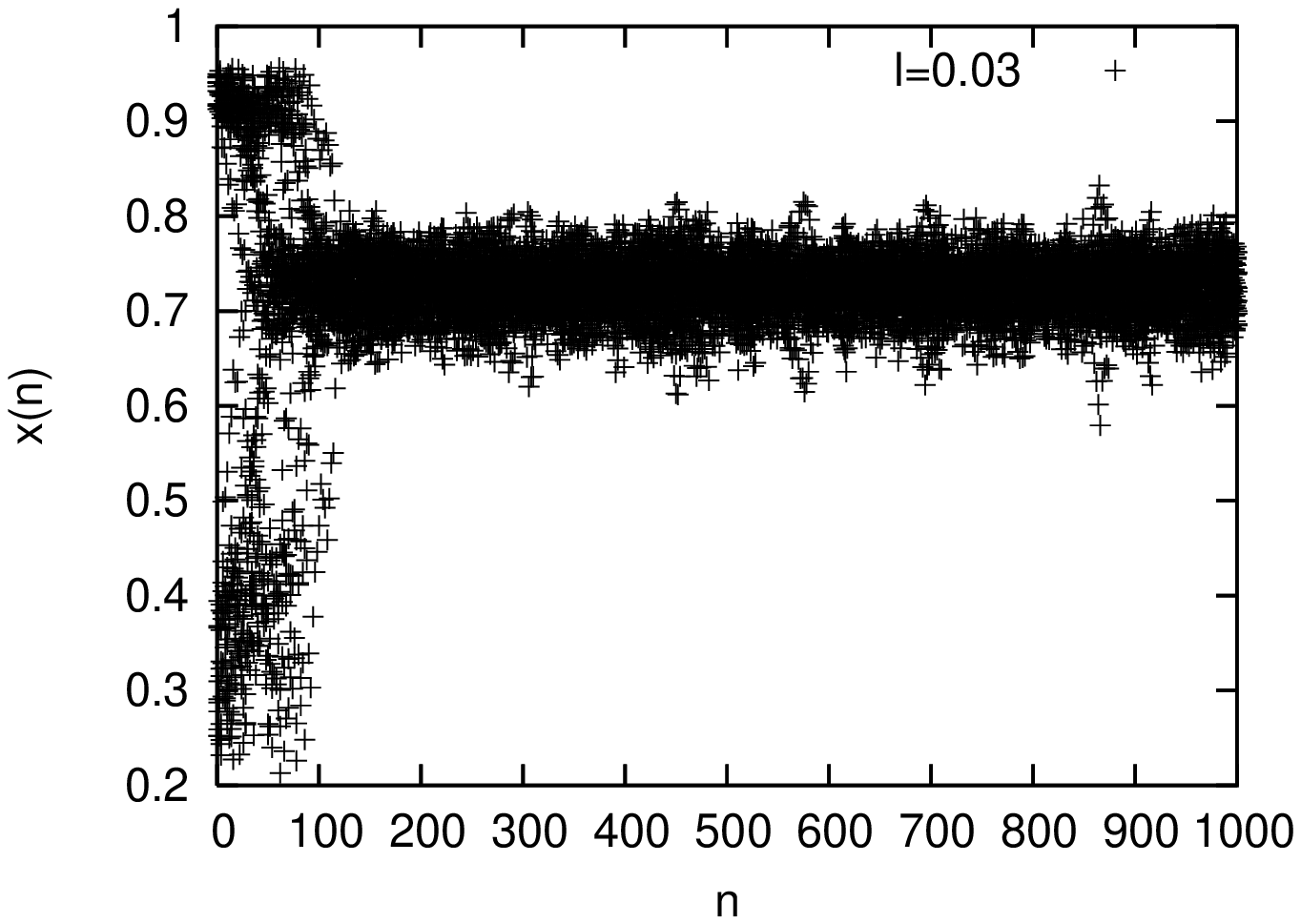}
\caption{Ten runs of the difference equation (\protect{\ref{pert_1}})  with $f$ as in (\protect{\ref{eq:S3.1}}) and
multiplicative stochastic perturbations with (left) $l=0.02$, $x_0=0.3217$ and (right) $l=0.03$,
$x_0=0.3217591$.}
\label{figure5}
\end{figure}

\end{example}  

The theoretical results of the present paper and the numerical simulations of this section illustrate the following conclusions:
\begin{enumerate}
\item
As expected, in the presence of either multiplicative or additive stochastic perturbations,
the unique positive equilibrium can become blurred.
\item
However, for a class of maps that includes commonly occurring models of population dynamics, stochasticity can contribute to the stability of this 
equilibrium. First, the bounds of the control parameter for which any solution of the controlled system 
converges to this (blurred) equilibrium expand. The second relevant issue is that even in the case when
the positive equilibrium of the deterministic equation is not globally attractive, its blurred version can
become attractive under perturbations, as numerical examples illustrate.
\end{enumerate}

\end{document}